\documentclass[reqno]{amsart}
\usepackage[dvipsnames]{xcolor}
\usepackage{amsmath,amsthm,amssymb,mathrsfs,graphics,enumitem,ytableau,tcolorbox,diagbox,fullpage,float,blkarray,etoolbox,thmtools,palatino}
\usepackage{tikz,tikz-cd}
\usepackage{soul}
\usepackage{hyperref}
\setlist[itemize]{noitemsep, topsep=0pt}
\setlist[enumerate]{noitemsep, topsep=0pt}
\tikzset{>=latex}

\hypersetup{colorlinks, citecolor=red, filecolor=black, linkcolor=blue, urlcolor=blue}

\newcommand{\defterm}[1]{\boldmath\textbf{#1}\unboldmath} % for definitions; can be changed to italic if journal style so requires

\newcommand{\csfop}{\mathbf{X}}
\newcommand{\csf}[1]{\csfop_{#1}}

\newcommand{\gdpop}{\mathbf{G}}
\newcommand{\gdp}[1]{\gdpop_{#1}}			% use \gdp{T} for the GDP of a tree T
\newcommand{\gdpco}{g}					% GDP coefficients
\newcommand{\hdpop}{\mathbf{H}}
\newcommand{\hdp}[1]{\hdpop_{#1}}			% use \hdp{T} for the HDP of a tree T
\newcommand{\hdpco}{h}					% HDP coefficients
\newcommand{\stpop}{\mathbf{S}}
\newcommand{\stp}[1]{\stpop_{#1}}			% use \stp{T} for the STP of a tree T
\newcommand{\stpco}{s}					% STP coefficients

\newcommand{\uhdpop}{\mathbf{\bar H}}			% just the symbol; don't use directly (except in one figure)
\newcommand{\uhdp}[1]{\uhdpop_{#1}}			% use \uhdp{T} for the modified HDP without end terms
\newcommand{\uhdpsup}[2]{\uhdpop^{#1}_{#2}}	% use, e.g., \uhdpsup{+v,-w}{T} for modified HDP with qualifiers

\newcommand{\soupop}{\mathbf{Q}}
\newcommand{\soup}[1]{\soupop_{#1}}

\DeclareMathOperator{\Cat}{Cat} % caterpillar with specified signature
\DeclareMathOperator{\spine}{sp}
\DeclareMathOperator{\type}{type}

\newcommand{\SSS}{\mathcal{S}}

\newcommand{\Hone}{\eta} % hdp of a caterpillar ; used to be h
\newcommand{\Htwo}{\zeta} % specialization of \Hone; used to be \mathcal{H}

\declaretheorem[style=plain]{theorem}
\declaretheorem[style=plain,sibling=theorem]{lemma}
\declaretheorem[style=plain,sibling=theorem]{corollary}

\declaretheorem[style=plain,sibling=theorem]{proposition}

\declaretheorem[style=definition,qed=$\blacktriangleleft$,sibling=theorem]{example}

\declaretheorem[style=definition]{question}

\newcommand{\0}{\emptyset}

\newcommand{\isom}{\cong}

\newcommand{\partn}{\vdash}

\newcommand{\sm}{\setminus}
\newcommand{\x}{\times}

\newcommand{\Nn}{\mathbb{N}}

\newcommand{\Qq}{\mathbb{Q}}

\newcommand{\Zz}{\mathbb{Z}}

\makeatletter
\def\moverlay{\mathpalette\mov@rlay}
\def\mov@rlay#1#2{\leavevmode\vtop{%
   \baselineskip\z@skip \lineskiplimit-\maxdimen
   \ialign{\hfil$\m@th#1##$\hfil\cr#2\crcr}}}
\newcommand{\charfusion}[3][\mathord]{
    #1{\ifx#1\mathop\vphantom{#2}\fi
        \mathpalette\mov@rlay{#2\cr#3}
      }
    \ifx#1\mathop\expandafter\displaylimits\fi}
\makeatother

\newcommand{\cupdot}{\charfusion[\mathbin]{\cup}{\cdot}}

\newcommand{\dju}{\cupdot}

\author[JAP]{Jos\'e Aliste-Prieto}
\address{Departamento de Matem\'aticas, Facultad de Ciencias Exactas, Universidad Andres Bello, Santiago, Chile}
\email{jose.aliste@unab.cl}
\author[JLM]{Jeremy L.\ Martin}
\address{Department of Mathematics, University of Kansas, Lawrence, KS, United States}
\email{jlmartin@ku.edu}
\author[JDW]{Jennifer D.\ Wagner}
\address{Department of Mathematics and Statistics, Washburn University, Topeka, KS, United States}
\email{jennifer.wagner1@washburn.edu}
\author[JZ]{Jos\'e Zamora}
\address{Departamento de Matem\'aticas, Facultad de Ciencias Exactas, Universidad Andres Bello, Santiago, Chile}
\email{josezamora@unab.cl}
\title{Chromatic symmetric functions and polynomial invariants of trees}
\thanks{The first and fourth authors acknowledge financial support from ANID/CONICYT FONDECYT Regular 1241663. The second and third authors thank their home institutions for granting sabbatical leave in Fall 2023, and Universidad Andres Bello in Santiago, Chile, for hosting them during that time, when this research was carried out.}
\date{\today}
 % symmetric functions and generalizations
\begin{document}
\begin{abstract}
Stanley asked whether a tree is determined up to isomorphism by its chromatic symmetric function.
We approach Stanley's problem by studying the relationship between the chromatic symmetric function and other invariants.
First, we prove Crew's conjecture that the chromatic symmetric function of a tree determines its generalized degree sequence, which enumerates vertex subsets by cardinality and the numbers of internal and external edges.
Second, we prove that the restriction of the generalized degree sequence to subtrees contains exactly the same information as the subtree polynomial, which enumerates subtrees by cardinality and number of leaves.  
Third, we construct arbitrarily large families of trees sharing the same subtree polynomial, proving and generalizing a conjecture of Eisenstat and Gordon.
\end{abstract}
\maketitle

\section{Introduction}

The \defterm{chromatic symmetric function} $\csf{G}$ of a graph $G$ enumerates proper colorings of~$G$ by their distributions of colors.
Introduced by Stanley in~\cite{Stanley-CSF}, the chromatic symmetric function (henceforth CSF) is a far-reaching generalization of the classical chromatic polynomial introduced by Birkhoff~\cite{Birkhoff}.
It is closely related to other important invariants in algebraic combinatorics, including Noble and Welsh's \defterm{U-polynomial} \cite{NobleWelsh}, whose definition is motivated by knot theory, and Billera, Jia and Reiner's \defterm{quasisymmetric function of matroids} \cite{BJR}.  The CSF plays a key role in the theory of combinatorial Hopf algebras, arising as the canonical morphism from the chromatic Hopf algebra of graphs to quasisymmetric functions \cite[\S4.5]{ABS}.  It has natural analogues in noncommutative symmetric functions \cite{GebSag} and quasisymmetric functions \cite{ShW}, with applications including the cohomology of Hessenberg subvarieties of flag manifolds.  Variations of the CSF have been developed for directed graphs, rooted trees, etc.: see, e.g., \cite{aliste2017trees,Ellzey,AWvW,aval2022quasisymmetric,Pawlowski,LW}.

Stanley's original article posed the question of whether the CSF is a complete isomorphism invariant for trees, i.e., whether $\csf{T}=\csf{T'}$ implies $T\isom T'$.  The problem remains unsolved despite considerable attention.  There is neither an easy way to construct two nonisomorphic trees with the same CSF, nor to extract sufficient local information to reconstruct a tree from the global data encoded in the CSF.
The distinguishing power of the CSF remains mysterious, in sharp contrast to weaker invariants like the chromatic polynomial (which provides no information about a tree other than the number of vertices) or versions of the CSF for labeled graphs (e.g., the noncommutative CSF of labeled trees studied by Gebhard and Sagan \cite{GebSag} is easily seen to be a complete invariant).  The uniqueness problem is also well understood for non-tree graphs: Stanley's original paper gave two five-vertex graphs with the same CSF, and Orellana and Scott \cite{OS} constructed an infinite family of pairs of unicyclic graphs (i.e., graphs that can be made into trees by deleting one edge) with the same CSF.  Additional families of graphs with the same CSFs are constructed in \cite{APCSZ}.

Partial progress has been made on Stanley's question.
Heil and Ji \cite{HeilJi} have verified that the CSF is a complete invariant for trees with up to 29 vertices.  The conjecture is also known for certain special classes of trees, including spiders~\cite{MMW}, caterpillars \cite{APZ,loebl2018isomorphism}, 2-spiders~\cite{huryn2020few}, and proper trees with diameter at most five~\cite{APdMOZ}.

Our approach to Stanley's problem is to compare the CSF to other graph invariants.  In particular, proving that a particular invariant can be computed from the CSF gives indirect evidence for the positive answer to Stanley's question and enables us to leverage the information from other invariants to learn more about what structural information about a tree is encoded in its CSF.  As an example of this approach, Martin, Morin and Wagner \cite{MMW} proved that the CSF determines the \defterm{subtree polynomial} $\stp{T}$, first studied by Chaudhary and Gordon \cite{CG} and Eisenstat and Gordon \cite{EG}, which enumerates subtrees of~$T$ by size and number of leaves.  In particular, the result of~\cite{MMW} implies that two elementary invariants of a tree, its degree and path sequences, are recoverable from its CSF, which is far from obvious from first principles. Similarly, Aliste--Prieto and Zamora \cite{APZ} proved that proper caterpillars are distinguished by their CSFs by characterizing graphs with the same \defterm{$\mathcal{L}$-polynomial}, a specialization of the U-polynomial of Noble and Welsh \cite{NobleWelsh} that is also obtainable from the CSF.

The question of which graphs are determined by a given polynomial invariant has been studied in several contexts: see, \emph{e.g.} \cite{noy2003graphs,mier2004graphs}. According to Tutte \cite{tutte1974codichromatic}, the problem of finding non-isomorphic graphs with the same value of a particular invariant can be traced back at least to unpublished work of Marion Gray in the 1930's, who found a pair of non-isomorphic graphs with the same doubly-indexed Whitney numbers (for which see \cite{whitney1932coloring}).  Finally, there has been significant effort to build a theory of graph polynomials that distinguish graphs: see, e.g., \cite{makowsky2008zoo,ellis2016graph} and references therein.

In this paper, we describe exactly the hierarchy among several polynomial invariants of trees, focusing on the three-variable \defterm{generalized degree polynomial} $\gdp{T}$ introduced by Crew \cite{Crew-note}; the \defterm{half-generalized degree polynomial}, a two-variable specialization $\hdp{T}$ of $\gdp{T}$, studied by Wang, Yu and Zhang \cite{WYZ}; and the two-variable Eisenstat-Gordon \defterm{subtree polynomial} $\stp{T}$ mentioned previously.
All of these invariants, which we will define precisely in Section~\ref{sec:background}, can be obtained from the CSF.

Our results are as follows:

First, we prove (Theorem~\ref{thm:crew}) that the CSF of a tree linearly determines its generalized degree polynomial linearly.  This result was conjectured by Crew in \cite{Crew-note}, and strengthens the result of Wang, Yu and Zhang \cite{WYZ} that $\csf{T}$ determines $\hdp{T}$.
Our proof of Theorem~\ref{thm:crew} is completely explicit: we exhibit an integer matrix that transforms the vector of coefficients of $\csf{T}$, written in the power-sum basis, into the vector of coefficients of $\gdp{T}$.  The entries of the matrix are purely combinatorial and depend only on the number of vertices of $T$.  The proof is given in Section~\ref{sec:crew}.

Second, we prove (Theorem~\ref{thm:sub-hdp}) that the polynomials $\hdp{T}$ and $\stp{T}$ are related by an invertible linear transformation, hence contain the same information.
To prove this result, we construct square integer matrices $\mathsf{M},\mathsf{N}$, invertible over $\Qq$, such that $\mathsf{M}\mathsf{H}=\mathsf{N}\mathsf{S}$, where $\mathsf{H},\mathsf{S}$ are the vectors of coefficients of $\hdp{T}$ and $\stp{T}$ respectively.  This proof is somewhat less combinatorial than that of Theorem~\ref{thm:crew} in the sense that there are no direct combinatorial interpretations for the entries of the transition matrices $\mathsf{N}^{-1}\mathsf{M}$ and $\mathsf{M}^{-1}\mathsf{N}$ (indeed, $\mathsf{N}$ is not invertible over $\Zz$).  The proof is given in Section~\ref{sec:HDP-STP}.
Observe that Theorems~\ref{thm:crew} and~\ref{thm:sub-hdp} together recover and strengthen the main result of~\cite{MMW}.

Third, we show (Theorem~\ref{thm:equiv-class}) how to construct arbitrarily large families of non-isomorphic trees sharing the same half-generalized degree polynomial, or, equivalently, the same subtree polynomial.  The trees in question are all \defterm{caterpillars}, which are indexed by integer compositions.  Billera, Thomas and van~Willigenburg~\cite[Theorem 3.6]{BTvW} showed that every composition $\alpha$ admits a certain unique irreducible factorization of the form $\alpha_1\circ\cdots\circ\alpha_k$ (the operation $\circ$ is defined in Section~\ref{sec:same-HDP} below).  We show that replacing any of the $\alpha_i$'s with their reversals produces a caterpillar with the same half-generalized degree polynomial.
This result proves and strengthens a conjecture of Eisenstat and Gordon~\cite[Conjecture~2.8]{EG} on constructing pairs of caterpillars with the same subtree polynomial.  The main tool in the proof is a recurrence for the half-generalized degree polynomial of a tree, using the operation of \defterm{near-contraction} introduced in~\cite{APdMOZ}.  The recurrence is proved in Section~\ref{sec:recurrence} and the construction of half-generalized degree polynomial equivalence classes in Section~\ref{sec:same-HDP}.

The relationships between these invariants are depicted in Figure~\ref{fig:relationships}.

\begin{figure}[t]
\begin{center}
\begin{tikzpicture}
\newcommand{\dy}{2.5}
\draw[->](0,-.75)--(0,-\dy+.75);

\node[rectangle,draw] at (0,0) {\begin{tabular}{c}
chromatic symmetric function\\
$\csf{T}$ (Stanley) \end{tabular}};

\draw[dashed,->] (2.75,0) -- (4.5,-2*\dy+.75); \node at (4.75,-\dy+.375) {\cite{MMW}};

\draw[dashed,->] (-2.75,0) -- (-5,-2*\dy+.75); \node at (-4.75,-\dy+.375) {\cite{WYZ}};

\node[rectangle,draw] at (0,-\dy) {\begin{tabular}{c}
generalized degree polynomial\\
$\gdp{T}$ (Crew) \end{tabular}};

\draw[dashed,->] (0,-\dy-.75) -- (-3,-2*\dy+.75); \node at (0,-1.5*\dy) {\begin{tabular}{c}strict (Tang)\end{tabular}};

\node[rectangle,draw] at (-4,-2*\dy) {\begin{tabular}{c}
half-generalized degree polynomial\\
$\hdp{T}$ (Wang--Yu--Zhang) \end{tabular}};

\draw[<->](-.5,-2*\dy)--(1.5,-2*\dy);

\node[rectangle,draw] at (4,-2*\dy) {\begin{tabular}{c}
subtree polynomial\\
$\stp{T}$ (Eisenstat--Gordon) \end{tabular}};

\end{tikzpicture}
\end{center}
\caption{Relationships between isomorphism invariants of trees.
The invariant at the tail of an arrow determines the invariant at the head. 
Solid arrows are results of this paper; dotted arrows indicate previously known results.}\label{fig:relationships}
\end{figure}
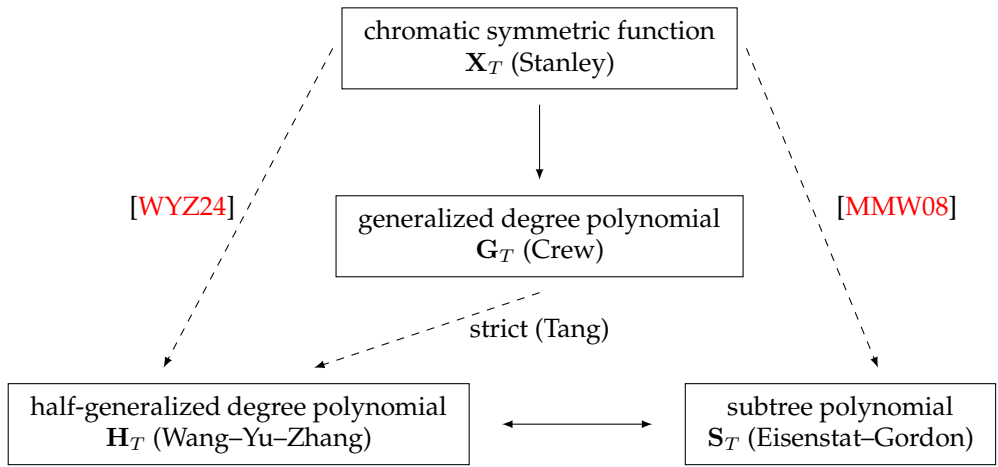

In the final section of the paper, we discuss possible directions for further research, including extending the constructions to non-caterpillar trees; comparison of the distinguishing power of the generalized and half-generalized degree polynomials; and another polynomial that simultaneously generalizes the half-generalized degree polynomial and the subtree polynomial.

In an earlier version of this paper, we observed that we could not certify that the GDP is a strictly stronger invariant than the HDP, based on exhaustive computation of these invariants for all trees on 18 or fewer vertices.  Subsequently, Michael Tang [personal communication] discovered two trees on 19 vertices with the same HDP but different GDPs.  More details are given in Section~\ref{gdp-and-hdp}.

\section{Background} \label{sec:background}

We make use of standard definitions and basic facts from graph theory and combinatorics; see, e.g., \cite{EC}.  A \defterm{tree} is a connected acyclic graph $T=(V,E)$ with at least one vertex.  A \defterm{subtree} $S$ of a tree $T$ is a connected subgraph whose vertex set $V(S)$ is nonempty.  The set of subtrees of $T$ is denoted $\SSS(T)$.

Let $n$ be a nonnegative integer.  A \defterm{partition} of $n$ is a non-increasing sequence $\lambda=(\lambda_1,\dots,\lambda_k)$ of positive integers that add up to $n$; in this case we write $\lambda\partn n$.  The number $k=\ell(\lambda)$ is the \defterm{length} of $\lambda$ and the numbers $\lambda_i$ are its \defterm{parts}.

A \defterm{symmetric function} is a formal power series in commuting variables $x_1,x_2,\dots$, say with coefficients in $\mathbb{Q}$, that is invariant under permutations of the $x_i$'s.  The set $\Lambda_n$ of all homogeneous degree-$n$ symmetric functions is a vector space of dimension equal to the number of partitions of~$n$.  For $\lambda=(\lambda_1,\dots,\lambda_k)\partn n$, the \defterm{monomial symmetric function} $m_\lambda$ is the sum of all monomials of the form $x_{i_1}^{\lambda_1}\cdots x_{i_k}^{\lambda_k}$, where the $i_j$ are distinct integers.  The \defterm{power-sum symmetric function} is $p_\lambda=\prod_{i=1}^k\left(\sum_{i\geq1} x_i^{\lambda_i}\right)$.  Both $\{m_\lambda\colon\lambda\partn n\}$ and $\{p_\lambda\colon\lambda\partn n\}$ are vector space bases for $\Lambda_n$.  (There are many other important bases of the symmetric functions, but these are the two that are relevant for our purposes.) For further details on symmetric functions, the reader is referred to, e.g., \cite[Chapter~7]{ECTwo}.

\subsection{The chromatic symmetric function}

Let $G=(V,E)$ be a simple graph with $|V|=n$. A \defterm{proper coloring} of $G$ is a map $\kappa:V\rightarrow\Nn_{>0}$ such that $\kappa(u)\neq\kappa(v)$ whenever $uv\in E$.
The \defterm{chromatic symmetric function}, introduced by Stanley \cite{Stanley-CSF}, is the formal power series
\[\csf{G}=\sum_{\substack{\text{proper colorings}\\ \kappa:V\to\Nn_{>0}}} \prod_{v\in V} x_{\kappa(v)},\]
which is a symmetric function in commuting variables $x_1,x_2,\dots$, homogeneous of degree~$n$.  Observe that setting $x_1=\cdots=x_k=1$ and $x_i=0$ for all $i>k$ recovers the number of proper colorings $\kappa:V\to[k]$, which is the chromatic polynomial of $G$ evaluated at $k$.

Consider the expansion of the chromatic symmetric function of a graph $G$ in the power-sum basis:
\[\csf{G}=\sum_{\lambda\,\partn\,n} c_\lambda(G) p_\lambda.\]
Stanley gave a combinatorial description of the numbers $c_\lambda(G)$ \cite[Thm.~2.5]{Stanley-CSF}, which is cancellation-free if and only if $G=T$ is a tree.  In that case, as observed in \cite[eqn.~(7)]{MMW}, one has
\begin{equation} \label{c-lambda}
c_\lambda(T)=(-1)^{n-\ell(\lambda)}|\{F\subseteq E\colon \type(F)=\lambda\}|.
\end{equation}
Here $\type(F)$ is the partition of $|V|$ whose parts are the sizes of the connected components of the graph $(V,F)$.
Thus we can regard the numbers $c_\lambda(T)$ as graph invariants derivable from the chromatic symmetric function.

\begin{example} \label{running-example}
The two trees on four vertices are the path $P_4$ and the star $S_4$, shown below.
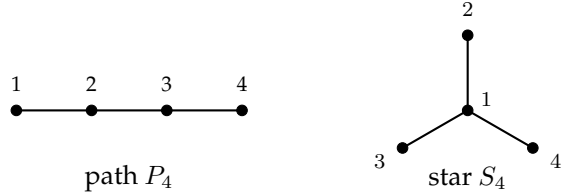
\begin{center}
\begin{tikzpicture}
\foreach \x in {1,2,3,4} { \draw[fill=black] (\x,0) circle (.07); \node at (\x,.35) {\footnotesize\x}; }
\draw[thick] (1,0)--(4,0);
\node at (2.5,-.9) {path $P_4$};
\begin{scope}[shift={(7,0)}]
\draw[fill=black] (0,0) circle (.07); \node at (.25,.15) {\footnotesize $1$};
\foreach \a/\b in {90/2,210/3,330/4} { \draw[fill=black] (\a:1) circle (.07); \draw[thick] (0,0)--(\a:1); \node at (\a:1.35) {\footnotesize $\b$}; }
\node at (0,-.9) {star $S_4$};
\end{scope}
\end{tikzpicture}
\end{center}

In the monomial basis, their chromatic symmetric functions are
\begin{align*}
\csf{P_4} &= 24m_{1111} + 6m_{211} + 2m_{22},\\
\csf{S_4} &= 24m_{1111} + 6m_{211} + m_{31}.
\end{align*}
For instance, there are two proper colorings of $P_4$ with two red and two blue vertices, but no such proper colorings for $S_4$.
On the other hand, $S_4$ admits one proper coloring with one red and three blue vertices, while $P_4$ has no such proper coloring.
The expansions of these chromatic symmetric functions in the power-sum basis are
\begin{align*}
\csf{P_4} &= p_{1111} -3p_{211} + p_{22} + 2p_{31} -p_4,\\
\csf{S_4} &= p_{1111} -3p_{211} + 3p_{31} -p_4.
\end{align*}
For instance, the coefficient $c_{22}(P_4)$ is 1 because the edge set $\{12,34\}$ induces a subgraph with two components of size~2, while $c_{22}(S_4)=0$ because the star has no such pair of edges.
Observe that both trees have the same chromatic polynomial: each has $k(k-1)^3$ proper colorings with $k$ colors.
\end{example}

\subsection{The generalized and half-generalized degree polynomials}

Let $T=(V,E)$ be a tree with $|V|=n$.  For a vertex set $A\subseteq V$, define
\begin{align*}
E(A) &= \{\text{edges of $E$ with both endpoints in $A$}\}, && e(A) = |E(A)|,\\
D(A) &= \{\text{edges of $E$ with exactly one endpoint in $A$}\}, && d(A) = |D(A)|,
\end{align*}
and let
\begin{equation} \label{gdp-coeffs}
\gdpco_T(a,b,c) = |\{A\subseteq V(T)\colon |A|=a,\ d(A)=b,\ e(A)=c\}|.
\end{equation}
The \defterm{generalized degree polynomial} (or GDP) of $T$, introduced by Crew \cite{Crew-thesis}, is
\begin{equation} \label{gdp}
\gdp{T} = \gdp{T}(x,y,z) = \sum_{A\subseteq V} x^{|A|} y^{d(A)} z^{e(A)} = \sum_{a,b,c} \gdpco_T(a,b,c) x^a y^b z^c.
\end{equation}
The reason for the name is that when $A=\{v\}$, the number $d(A)$ is just the degree of vertex~$v$ (i.e., the number of edges incident to~$v$). Crew \cite[\S4.3]{Crew-thesis} introduced the \defterm{generalized degree sequence} of $T$ as the multiset of triples $(|A|,d(A),e(A))$ for $A\subseteq V$; our generating function is equivalent.

For $A\subseteq V$, let $T|_A$ denote the subgraph induced by $A$, and define
\begin{equation} \label{hdp-coeffs}
\hdpco_T(b,c) = |\{A\subseteq V\colon T|_A\text{ connected},\ d(A)=b,\ e(A)=c\}| = \gdpco_T(c+1,b,c).
\end{equation}
The \defterm{half-generalized degree polynomial} (or HDP) of $T$ is defined as
\begin{equation}\label{hdp}
\begin{aligned}
\hdp{T} = \hdp{T}(y,z) &= \sum_{\substack{\0\neq A\subseteq V\\ T|_A\text{ connected}}} y^{d(A)} z^{e(A)}
&= \sum_{S\in\SSS(T)} y^{d(S)} z^{e(S)} = \sum_{b,c} \hdpco_T(b,c) y^b z^c.
\end{aligned}
\end{equation}
Evidently, if $|A|=0$, then $(d(A),e(A))=(0,0)$, and if $|A|=n$, then $(d(A),e(A))=(0,n-1)$.  Otherwise, if $1\leq|A|\leq n-1$, then:
\begin{itemize}
\item $1\leq d(A)\leq n-1$;
\item $0\leq e(A)<|A|$; and
\item $|A|\leq d(A)+e(A)\leq n-1$.
\end{itemize}
The only inequality that is not immediate is $d(A)+e(A)\geq|A|$.  To see this, observe that the forest $T|_A$ has $e(A)$ edges and $|A|-e(A)$ components.  Each component $K$ has at least one edge with one endpoint in $K$ and one in $V(T)\sm A$.  These edges are all distinct external edges, so $|A|-e(A)\leq d(A)$.  Accordingly, we can write the generalized and half-generalized degree polynomials as
\begin{align*}
\gdp{T}(x,\,y,\,z)&=1+x^nz^{n-1}+\sum_{a=1}^{n-1} \sum_{c=0}^{a-1} \sum_{b=1}^{n-1-c} \gdpco_T(a,b,c) x^a y^b z^c,\\
\hdp{T}(y,\,z)&=z^{n-1}+\sum_{c=0}^{n-1} \sum_{b=1}^{n-1-c} \hdpco_T(b,c) y^b z^c.
\end{align*}
Crew \cite{Crew-note} conjectured that $\gdp{T}$ can be recovered from $\csf{T}$.  Wang, Yu, and Zhang \cite[Thm.~5.3]{WYZ} proved that $\hdp{T}$ can be recovered linearly from $\csf{T}$; that is, the coefficients $\hdpco_T(b,c)$ of the HDP are linear functions of the coefficients $c_\lambda(T)$ of~\eqref{c-lambda}.

\begin{example} \label{running-example:2}
Consider the star $S_4$, with vertices labeled as in Example~\ref{running-example}.
Its generalized and half-generalized degree polynomials are computed from the definitions~\eqref{gdp} and~ \eqref{hdp} as follows
(vertex sets are abbreviated, e.g., 23 for $\{2,3\}$).
\[
\begin{array}{c|cccccccc}
A & \0 & 1 & 2,3,4 & 12,13,14 & 23,24,34 & 123,124,134 & 234 & 1234 \\ \hline
|A| & 0 & 1&1 & 2&2 & 3&3 & 4\\
d(A) & 0 & 3&1 & 2&2 & 1&3 & 0\\
e(A) & 0 & 0&0 & 1&0 & 2&0 & 3\\ \hline
\rule{0bp}{10bp} % to add space below the \hline, improving readability
\gdp{S_4} = & 1 & +\:xy^3 &+\: 3xy &+\:3x^2y^2z &+\: 3x^2y^2 &+\:3x^3yz^2 &+\:x^3y^3 &+\:x^4z^3  \\
\hdp{S_4} = &  & y^3 &+\: 3y &+\:3y^2z & &+\:3yz^2 & &+\:z^3
\end{array}
\]
For the path $P_4$, the computation is as follows:
\[
\begin{array}{c|cccccccccc}
A & \0 & 1,4 & 2,3 & 12,34 & 23 & 13,24 & 14 & 123,234 & 124,134 & 1234 \\ \hline
|A| & 0 & 1&1 & 2&2&2&2 & 3&3 & 4\\
d(A) & 0 & 1&2 & 1&2&3&2 & 1&2 & 0\\
e(A) & 0 & 0&0 & 1&1&0&0 & 2&1 & 3\\ \hline
\rule{0bp}{10bp} % to add space below the \hline, improving readability
\gdp{P_4} = & 1 & +\:2xy &+\: 2xy^2 &+\:2x^2yz &+\: x^2y^2z &+\:2x^2y^3 &+\:x^2y^2 &+\:2x^3yz^2 &+\:2x^3y^2z &+\:x^4z^3\\
\hdp{P_4} = &  & 2y &+\: 2y^2 &+\:2yz & +\:y^2z & & &+\:2yz^2 & & +\:z^3
\end{array}
\]
\end{example}

As observed by Crew \cite[pp.~83--84]{Crew-thesis}, the generalized degree polynomial is not a complete invariant for trees.  The two smallest trees with the same GDP are shown in Figure~\ref{fig:sameGDP}.  These are also the smallest trees with the same HDP.

\begin{figure}[th]
\begin{center}
\begin{tikzpicture}
\foreach \x in {0,1,2,3,4} \draw[fill=black] (\x,0) circle (.08);
\draw(0,0)--(4,0);
\foreach \x/\s in {0/0, 1/1, 2.8/3, 3.2/3, 3.8/4, 4.2/4} { \draw (\x,-1)--(\s,0);  \draw[fill=black] (\x,-1) circle (.08); }
\begin{scope}[shift={(8,0)}]
\foreach \x in {0,1,2,3,4} \draw[fill=black] (\x,0) circle (.08);
\draw(0,0)--(4,0);
\foreach \x/\s in {0/0, .8/1, 1.2/1, 2/2, 3.8/4, 4.2/4} { \draw (\x,-1)--(\s,0);  \draw[fill=black] (\x,-1) circle (.08); }
\end{scope}
\end{tikzpicture}
\end{center}
\caption{The smallest trees with the same generalized degree polynomial and the same subtree polynomial.\label{fig:sameGDP}}
\end{figure}
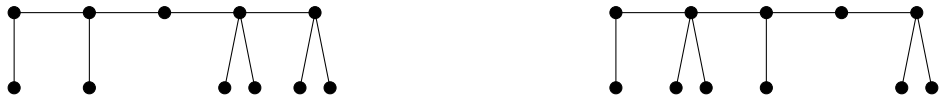

\subsection{The subtree polynomial}
Let $T=(V,E)$ be a tree with $|V|=n$, and let $S\in\SSS(T)$.  An edge of $S$ is called a \defterm{leaf edge} if at least one of its endpoints is a leaf of~$S$.
Let
\begin{align*}
E(S)&=E(V(S)) & e(S) &= |E(S)|,\\
L(S)&=\{\text{leaf edges of $S$}\}, & \ell(S) &= |L(S)|.
\end{align*}
and define
\begin{equation} \label{stp-coeffs}
\stpco_T(i,j) = |\{U\in\SSS(T)\colon e(U)=i,\ \ell(U)=j\}|.
\end{equation}
The \defterm{subtree polynomial} (or STP) of $T$, introduced by Eisenstat and Gordon~\cite{EG} (and equivalent to the \defterm{greedoid Tutte polynomial} introduced in \cite{GMcM}) is defined as
\begin{equation}\label{stp}
\stp{T}(q,r) = \sum_{S\in\SSS(T)} q^{e(S)} r^{\ell(S)} = \sum_{i,j} \stpco_T(i,j) q^i r^j.
\end{equation}
Martin, Morin and Wagner \cite{MMW} proved that the chromatic symmetric function determines the subtree polynomial linearly.  The STP is not a complete tree invariant; as for the GDP and HDP, the two trees shown in Figure~\ref{fig:sameGDP} are the smallest pair with the same STP.

\begin{example} \label{running-example:3}
Consider the path $P_4$, with vertices labeled as in Example~\ref{running-example}.
Its subtree polynomial is computed from the definition~\eqref{stp} as follows, where each subtree $S$ is indicated by its vertex set $V(S)$.
\[
\begin{array}{c|cccc}
V(S) & 1,2,3,4 & 12,23,34 & 123,234 & 1234 \\ \hline
e(S) & 0 & 1 & 2 & 3\\
\ell(S) & 0 & 1 & 2 & 2\\\hline
\rule{0bp}{10bp} % to add space below the \hline, improving readability
\stp{P_4} = & 4 &+\: 3qr &+\:2q^2r^2 &+\:q^3r^2
\end{array}
\]
For the star $S_4$, the computation is as follows:
\[
\begin{array}{c|cccc}
V(S) & 1,2,3,4 & 12,13,14 & 123,124,134 & 1234 \\ \hline
e(S) & 0 & 1 & 2 & 3\\
\ell(S) & 0 & 1 & 2 & 3\\\hline
\rule{0bp}{10bp} % to add space below the \hline, improving readability
\stp{S_4} = & 4 &+\: 3qr &+\:3q^2r^2 &+\:q^3r^3
\end{array}
\]
\end{example}

\section{The chromatic symmetric function determines the generalized degree polynomial} \label{sec:crew}

In this section we prove that the coefficients of the generalized degree polynomial of a tree are determined linearly by the coefficients of the power-sum expansion of its chromatic symmetric function.  This result proves Crew's conjecture from~\cite{Crew-note}.

Throughout this section, let $T=(V,E)$ be a tree with $|V|=n$.  For $F\subseteq E$ and $A\subseteq V$, write $F(A)=F\cap E(A)$ and $F(\bar A)=F\cap E(\bar A)$, where $\bar A=V\sm A$.  Say that $A$ is \defterm{$F$-pure} if $A$ is a union of vertex sets of connected components of the graph $(V,F)$. Equivalent conditions are $F\cap D(A)=\0$ and $F\subseteq E(A)\cup E(\bar A)$.  More specifically, say that $A$ is \defterm{$F$-pure of type $\mu$} if $A$ is a union of vertex sets of components of $F$ whose sizes are the parts of $\mu$ (so, in particular, $|A|=|\mu|$).

\begin{example}
Let $G$ be the graph shown below and let $A=\{1,2,5\}$.
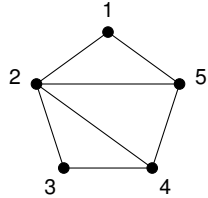
\begin{center}
\begin{tikzpicture}
\foreach \a in {1,...,5} { \draw[fill=black] (18+\a*72:1) circle (.07); \node at (18+\a*72:1.3) {\footnotesize\sf\a}; }
\draw(18:1)--(90:1)--(162:1)--(234:1)--(306:1)--(18:1)--(162:1)--(306:1);
\end{tikzpicture}
\end{center}
  If $F=\{12\}$ or $F=\{12,34\}$, then $A$ is $F$-pure of type $(2,1)$, since $\{1,2\}$ and $\{5\}$ are vertex sets of components of $F$.  Likewise, if $F=\0$ or $F=\{12,15\}$, then $A$ is $F$-pure of type $(1,1,1)$ or $(3)$ respectively.  On the other hand, if $F=\{12,45\}$, then $A$ is not $F$-pure.\end{example}

For partitions $\lambda$ and $\mu$ define
\[\binom{\lambda}{\mu}:=\prod_{i=1}^n \binom{m_i(\lambda)}{m_i(\mu)}\]
where $m_i(\lambda)$ is the number of occurrences of~$i$ as a part of $\lambda$.
Observe that if $\type(F)=\lambda$, then the number of $F$-pure sets of type~$\mu$ is $\binom{\lambda}{\mu}$.

We will need the following simple combinatorial identity:
\begin{lemma} \label{useful-lemma}
For all sets $P$ and numbers $q$, we have
\[\sum_{F\subseteq P} (-1)^{|F|+q} \binom{|F|}{q} =
\begin{cases} 1 & \text{ if } |P|=q, \\ 0 & \text{ if } |P|\neq q.\end{cases}\]
\end{lemma}

\begin{proof}
Let $p=|P|$.  Using a standard binomial identity \cite[equation ~(5.21), p.167]{GKP}, we have
\begin{align*}
\sum_{k=0}^p (-1)^k \binom{p}{k} \binom{k}{q}
&= \sum_{k=0}^p (-1)^k \binom{p}{q} \binom{p-q}{k-q}
&= \binom{p}{q} \sum_{k=q}^p (-1)^k \binom{p-q}{k-q}
&= (-1)^q \binom{p}{q} \sum_{j=0}^{p-q} (-1)^{j} \binom{p-q}{j}.
\end{align*}
The sum is the binomial expansion of $(1-1)^{p-q}$.  In particular it vanishes unless $p=q$, in which case the entire expression is 1.
\end{proof}

For $\lambda\partn n$ and numbers $a,b,c$, define
\[\omega(\lambda,\,a,\,b,\,c) = (-1)^{n-b-1} \sum_{\mu\,\partn\,a} \binom{a-\ell(\mu)}{c}\binom{\lambda}{\mu} \binom{n-\ell(\lambda)+\ell(\mu)-a}{n-b-c-1}.\]

\begin{theorem}\label{thm:crew}
The chromatic symmetric function determines the generalized degree polynomial linearly.
\end{theorem}

\begin{proof}
We will prove that
\begin{equation} \label{unicorns-and-rainbows}
\gdpco_T(a,b,c) = \sum_{\lambda\,\partn\,n} c_\lambda(T) \omega(\lambda,a,b,c)
\end{equation}
where $\gdpco_T(a,b,c)$ is defined as in~\eqref{gdp-coeffs}.

Let $R$ be the right-hand side of~\eqref{unicorns-and-rainbows}.  We start by plugging in the definitions of $c_\lambda(T)$ and $\omega ( \lambda, a, b, c )$:
\begin{align*}
R
&= \sum_{\lambda\,\partn\,n} (-1)^{n-\ell(\lambda)}|\{F\subseteq E\colon \type(F)=\lambda\}| (-1)^{n-b-1} \sum_{\mu\,\partn\,a} \binom{a-\ell(\mu)}{c}\binom{\lambda}{\mu} \binom{n-\ell(\lambda)+\ell(\mu)-a}{n-b-c-1} \\
&= \sum_{F\subseteq E} (-1)^{n-\ell(\type(F))} (-1)^{n-b-1} \sum_{\mu\,\partn\,a} \binom{a-\ell(\mu)}{c}\binom{\type(F)}{\mu} \binom{n-\ell(\type(F))+\ell(\mu)-a}{n-b-c-1} \\
&= \sum_{F\subseteq E} (-1)^{|F|+n-b-1} \sum_{\mu\,\partn\,a} \binom{a-\ell(\mu)}{c}\binom{\type(F)}{\mu} \binom{|F|+\ell(\mu)-a}{n-b-c-1} \\
&= \sum_{F\subseteq E} (-1)^{|F|+n-b-1} \sum_{\mu\,\partn\,a} \sum_{\substack{A\in\binom{V}{a}\\ \text{$F$-pure of type $\mu$}}}
\binom{a-\ell(\mu)}{c} \binom{|F|+\ell(\mu)-a}{n-b-c-1} \\
&= \sum_{F\subseteq E} (-1)^{|F|+n-b-1} \sum_{\substack{A\in\binom{V}{a}\\ \text{$F$-pure}}}
\binom{a-\ell(\type(F(A)))}{c} \binom{|F|+\ell(\type(F(A)))-a}{n-b-c-1}
	\intertext{where $F(A)$ is regarded as an edge set on $A$, so that $\type(F(A))\partn|A|=a$.
	Now recall that $A$ is $F$-pure if and only if $F\subseteq E(A)\cup E(\bar A)$.
	So we may switch the order of summation to obtain}
R &=  \sum_{A\in\binom{V}{a}} \sum_{F\subseteq E(A)\cup E(\bar A)} (-1)^{|F|+n-b-1}
\binom{a-\ell(\type(F(A)))}{c} \binom{|F|+\ell(\type(F(A)))-a}{n-b-c-1}. \\
\intertext{Every set $F\subseteq E(A)\cup E(\bar A)$ can be written uniquely as $F=F(A)\dju F(\bar A)$, with $F(A)\subseteq E(A)$ and $F(\bar A)\subseteq E(\bar A)$.  Moreover, $a-\ell(\type(F(A)))=|F(A)|$.  So now we get}
R &=  \sum_{A\in\binom{V}{a}} \sum_{F(A)\subseteq E(A)} \sum_{F(\bar A)\subseteq E(\bar A)} (-1)^{|F(A)|+|F(\bar A)|+n-b-1}
\binom{|F(A)|}{c} \binom{|F(\bar A)|}{n-b-c-1} \\
&=  \sum_{A\in\binom{V}{a}}
	\left( \sum_{F(A)\subseteq E(A)} (-1)^{|F(A)|+c} \binom{|F(A)|}{c} \right)
	\left( \sum_{F(\bar A)\subseteq E(\bar A)} (-1)^{|F(\bar A)|+n-b-c-1} \binom{|F(\bar A)|}{n-b-c-1} \right).\\
\intertext{Now, applying Lemma~\ref{useful-lemma} twice, we get}
R &= \big|\{A\subseteq [n]\colon |A|=a,\ e(A)=c,\ e(\bar A)=n-b-c-1\}\big|
\end{align*}
and the theorem follows since $d(A)=n-1-e(A)-e(\bar A)$.
\end{proof}

An immediate consequence is a formula for the degree sequence of a tree.
\begin{corollary}
The number of vertices of degree $b$ in $T$ is
\begin{align*}
|\{A\subseteq [n]\colon |A|=1,\ d(A)=b,\ e(A)=0\}|
&= \sum_{\lambda\,\partn\,n} c_\lambda(T) (-1)^{n-b-1} \sum_{\mu\,\partn\,1} \binom{1-\ell(\mu)}{0}\binom{\lambda}{\mu} \binom{n-\ell(\lambda)+\ell(\mu)-1}{n-b-0-1}\\
&= \sum_{\lambda\,\partn\,n} c_\lambda(T) (-1)^{n-b-1} m_1(\lambda) \binom{n-\ell(\lambda)}{n-b-1}.
\end{align*}
\end{corollary}
By contrast, the results of \cite{MMW} imply that for $b\geq 2$, the number of vertices of degree $b$ is also given by the formula
\begin{align*}
\sum_{k\geq b}\binom{k}{b}(-1)^{k+b}\sigma_k
&=\sum_{k\geq b}\binom{k}{b}(-1)^{k+b} [q^kr^k] \stp{T}(q,r)\\
&=\sum_{k\geq b}\binom{k}{b}(-1)^{k+b} \sum_{\lambda\partn n} c_\lambda(T) \binom{\ell(\lambda)-1}{\ell(\lambda)-n+k} \sum_{d=1}^k (-1)^d \sum_{j=1}^{\ell(\lambda)}\binom{\lambda_j-1}{d}.
\end{align*}
We thank Michael Tang for pointing out to us that these two expressions can be shown to be equivalent using elementary methods.

We take this opportunity to point out a minor error in~\cite{MMW}: the last line of the proof of Corollary~5 therein should read ``for every $k\geq2$'', not ``for every $k\geq1$''.  The mistake does not affect the proof since the number of leaves of $T$ can easily be recovered from $\csf{T}$, for instance as $|c_{(n-1,1)}|$.

\section{The half-generalized degree polynomial and the subtree polynomial are equivalent} \label{sec:HDP-STP}

In this section we prove that the coefficients of the half-generalized degree polynomial of a tree, and those of its subtree polynomial, determine each other linearly.  

Throughout this section, let $T=(V,E)$ be a tree with $|V|=n$.  Let $\hdpco(a,b)=\hdpco_T(a,b)$ and $\stpco(i,j)=\stpco_T(i,j)$ be the coefficients of $\hdp{T}$ and $\stp{T}$ defined in~\eqref{hdp} and~\eqref{stp} respectively.  Define column vectors
\begin{align*}
\mathsf{H}_1&=[\hdpco(0,b)]_{b=0}^{n-1}, & \mathsf{S}_1&=[\stpco(k,k)]_{k=0}^{n-1},\\
\mathsf{H}_2&=[\hdpco(a,b)]_{1\leq a,b;\ a+b\leq n-1}, & \mathsf{S}_2&=[\stpco(i,j)]_{2\leq j\leq i\leq n-1}.
\end{align*}
Observe that $\hdp{T}$ is determined by the entries of $\mathsf{H}_1$ and $\mathsf{H}_2$ together,
because $\hdpco(a,b)=0$ for all other $(a,b)$, except $\hdpco(n-1,0)=1$.
Similarly, $\stp{T}$ is determined by the numbers $\stpco(0,0)=n$, $\stpco(1,1)=n-1$, and $\stpco(i,j)$ for $2\leq j\leq i\leq n-1$ 
, since $s_T(i,1)=0$ if $i>1$ and also $s_T(i,j)=0$ when  $i<j$.

\begin{theorem}\label{thm:sub-hdp}
The half-generalized degree polynomial and the subtree polynomial determine each other linearly.
Specifically, there exist nonsingular integer matrices $\mathsf{P},\mathsf{M},\mathsf{N}$ such that $\mathsf{P}\mathsf{H}_1=\mathsf{S}_1$ and $\mathsf{M}\mathsf{H}_2=\mathsf{N}\mathsf{S}_2$.
\end{theorem}

\begin{proof}
For a subtree $S\in\SSS(T)$, let $D(S)=D(V(S))$ (the set of edges with exactly one endpoint a vertex of $S$) and $d(S)= |D(S)|$.
For each $K\subseteq D(S)$, let $U=S\cup K$.  Then $U$ is a tree and every edge in $K$ is a leaf edge of~$U$.
Since $S$ can be recovered from the pair $(U,K)$, we have a bijection
\begin{align*}
\{(S,K)\colon S\in\SSS(T),\ K\subseteq D(S)\}
&\xrightarrow{\xi}
\{(U,K)\colon U\in\SSS(T),\ K\subseteq L(U)\}\\
(S,K) &\mapsto (S\cup K,K)
\end{align*}
with $\xi^{-1}(U,K)=(U\sm K,K)$.  For nonnegative integers $a,k$ with $a+k\leq n-1$, the map $\xi$ restricts to a bijection
\[
\{(S,K)\colon S\in\SSS(T),\ K\subseteq D(S),\ |S|=a,\ |K|=k\}
\to
\{(U,K)\colon U\in\SSS(T),\ K\subseteq L(U),\ |U|=a+k,\ |K|=k\}.
\]
Fixing $a$ and $k$ and summing over the possibilities for $b=d(S)$ and $j=\ell(U)$, we obtain equalities
\begin{equation} \label{master-equality}
\sum_{b=k}^{n-1-a} \binom{b}{k} \hdpco(a,b) = \sum_{j=k}^{n-1} \binom{j}{k} \stpco(a+k,j)
\end{equation}
for every $a,k$.
\medskip

\underline{Claim 1: The vectors $\mathsf{H}_1=[\hdpco(0,b)]_{b=0}^{n-1}$ and $\mathsf{S}_1=[\stpco(k,k)]_{k=0}^{n-1}$ determine each other.}
\medskip

Indeed, consider the $n$ equations~\eqref{master-equality} when $a=0$ and $0\leq k\leq n-1$: they are 
\[\sum_{b=k}^{n-1} \binom{b}{k} \hdpco(0,b) = \sum_{j=k}^{n-1} \binom{j}{k} \stpco(k,j) = \stpco(k,k)\]
(because if $j>k$ then $\stpco(k,j)=0$).  In matrix form, this system of equations can be written as $\mathsf{P}\mathsf{H}_1=\mathsf{S}_1$, where the matrix
\[\mathsf{P}=\left[\binom{j}{i}\right]_{i,j=0}^{n-1}\]
is evidently (uni)triangular, proving Claim~1.
\medskip

\underline{Claim 2: The vectors $\mathsf{H}_2=[\hdpco(a,b)]_{1\leq a,b;\ a+b\leq n-1}$ and $\mathsf{S}_2=[\stpco(i,j)]_{2\leq j\leq i\leq n-1}$ determine each other.}
\medskip

This time, consider the $\binom{n-1}{2}$ equations~\eqref{master-equality} when $a,k>0$ and $a+k\leq n-1$.   Then the data sets in the claim are exactly the variables appearing in the equations.  (Note that $\stpco(a+k,1)=0$, because $a+k\geq 2$ and every tree with at least two edges has at least two leaf edges.)
Therefore, the equations we are considering can be written in matrix form as $\mathsf{M}\mathsf{H}_2 = \mathsf{N}\mathsf{S}_2$, where $\mathsf{M}$ and $\mathsf{N}$ are $\binom{n-1}{2}\x\binom{n-1}{2}$ square matrices whose rows are indexed by the pairs $(a,k)$.

If we list the rows of $\mathsf{M}$ in increasing order by $a$, then in increasing order by $k$, then $\mathsf{M}$ has the block diagonal form $\mathsf{M}_2 \oplus\cdots\oplus \mathsf{M}_{n-1}$, where
\[\mathsf{M}_a = \left[ \binom{b}{k} \right]_{(b,k)\in[n-a]\x[n-a]}.\]
In particular, $\mathsf{M}_a$ is unitriangular, hence nonsingular, and so $\mathsf{M}$ is nonsingular as well (in fact invertible over $\Zz$).

On the other hand, if we list the rows of $\mathsf{N}$ in increasing order by $a+k$, then in increasing order by $a$, then $\mathsf{N}$ has the block diagonal form $\mathsf{N}_2\oplus\cdots\oplus \mathsf{N}_{n-1}$, where
\[\mathsf{N}_i = \left[\binom{k+1}{j-i}\right]_{(j,k)\in[i-1]\x[i-1]}.\] 
In particular, $\det \mathsf{N}_i$ is a binomial determinant in the sense of Gessel and Viennot~\cite{GV}, namely $\binom{2,\,3,\,\dots,\,i}{1,\,2,\,\dots,\,i-1}$ in the notation of that paper, and in particular \cite[Corollary~2]{GV} guarantees that $\det(\mathsf{N}_i)>0$.  In fact, a short calculation using \cite[Lemmas~8 and~9]{GV} shows that $\det \mathsf{N}_i=i$ for each $i$, so that $\det \mathsf{N}=n!$.
Thus we have shown that the coefficients of $\hdp{T}$ and $\stp{T}$ determine each other linearly.
\end{proof}

To illustrate the proof, for $n=5$,
the matrix equation $\mathsf{M}\mathsf{H}_2=\mathsf{N}\mathsf{S}_2$ can be written either as
\[
\left[\begin{array}{ccc|cc|c} 1&2&3&0&0&0 \\ 0&1&3&0&0&0 \\ 0&0&1&0&0&0 \\ \hline 0&0&0&1&2&0 \\ 0&0&0&0&1&0 \\ \hline 0&0&0&0&0&1 \end{array}\right]
\begin{bmatrix} \hdpco(1,1)\\ \hdpco(1,2)\\ \hdpco(2,1)\\ \hdpco(1,3)\\ \hdpco(2,2)\\ \hdpco(3,1) \end{bmatrix}
=
\begin{bmatrix} 2&0&0&0&0&0 \\ 0&1&3&0&0&0 \\ 0&0&0&0&1&4 \\ 0&2&3&0&0&0 \\ 0&0&0&1&3&6 \\ 0&0&0&2&3&4 \end{bmatrix}
\begin{bmatrix} \stpco(2,2)\\ \stpco(3,2)\\ \stpco(3,3)\\ \stpco(4,2)\\ \stpco(4,3)\\ \stpco(4,4) \end{bmatrix}
\]
illustrating the block diagonal form of $\mathsf{M}$, or as
\[
\begin{bmatrix} 1&2&3&0&0&0 \\ 0&1&3&0&0&0 \\ 0&0&0&1&2&0 \\ 0&0&1&0&0&0 \\ 0&0&0&0&1&0 \\ 0&0&0&0&0&1 \end{bmatrix}
\begin{bmatrix} \hdpco(1,1)\\ \hdpco(1,2)\\ \hdpco(2,1)\\ \hdpco(1,3)\\ \hdpco(2,2)\\ \hdpco(3,1) \end{bmatrix}
=
\left[\begin{array}{c|cc|ccc} 2&0&0&0&0&0 \\ \hline 0&1&3&0&0&0 \\ 0&2&3&0&0&0 \\ \hline 0&0&0&0&1&4 \\ 0&0&0&1&3&6 \\ 0&0&0&2&3&4 \end{array}\right]
\begin{bmatrix} \stpco(2,2)\\ \stpco(3,2)\\ \stpco(3,3)\\ \stpco(4,2)\\ \stpco(4,3)\\ \stpco(4,4) \end{bmatrix}
\]
illustrating the block diagonal form of $\mathsf{N}$. 

\begin{example} \label{running-example:4}
Continuing our running example, the path $P_4$ has $\mathsf{H}_2=(2,1,3)$ and $\mathsf{S}_2=(2,1,0)$,
and the star $S_4$ has $\mathsf{H}_2=(0,3,3)$ and $\mathsf{S}_2=(3,1,0)$ (obtained from the calculations in Examples~\ref{running-example:2} and~\ref{running-example:3}).
For $n=4$, the matrix equation $\mathsf{M}\mathsf{H}_2=\mathsf{N}\mathsf{S}_2$ is
\[
\begin{bmatrix} 1&2&0 \\ 0&1&0 \\ 0&0&1 \end{bmatrix}
\begin{bmatrix} \hdpco(1,1)\\ \hdpco(1,2)\\ \hdpco(2,1)\end{bmatrix}
=
\begin{bmatrix} 2&0&0 \\ 0&1&3 \\ 0&2&3 \end{bmatrix}
\begin{bmatrix} \stpco(2,2)\\ \stpco(3,2)\\ \stpco(3,3)\\ \end{bmatrix}
\]
and one can check that the equation holds for both the path and the star.
\end{example}

As another remark, the explicit linear transformations mapping $\hdp{T}$ and $\stp{T}$ to each other are given by the matrices $\mathsf{M}^{-1}\mathsf{N}$ and $\mathsf{N}^{-1}\mathsf{M}$.  We do not expect the entries of either of these matrices to have combinatorially meaningful formulas, particularly for the matrix $\mathsf{N}^{-1}\mathsf{M}$, whose entries are not integers. 

\section{A recurrence for the half-generalized degree polynomial} \label{sec:recurrence}

Let $T$ be a tree with $n\geq 3$ vertices, and let $\SSS'(T)$ denote the family of subtrees of $T$ that contain at least one non-leaf vertex of $T$.
It is convenient to work with a slight modification of the half-generalized degree polynomial, defined by
\begin{equation}\label{uhdp}
\uhdp{T} = \sum_{S\in\SSS'(T)} y^{d(S)}z^{e(S)} = \hdp{T}-\ell(T)y.
\end{equation}
For instance, if $T$ is the star $S_4$, then $\hdp{T} = y^3 + 3y + 3y^2 + 3yz^2 + z^3$ (see Example~\ref{running-example:2}), so
$\uhdp{T} = y^3 + 3y^2 + 3yz^2 + z^3$.
The polynomial $\uhdp{T}$ contains the same information as $\hdp{T}$, but is more convenient to work with for our present purposes. 

Let $e=vw$ be a non-leaf edge of $T$.  The \defterm{near-contraction} of $e$ \cite[\S3]{APdMOZ} is the tree $T\odot e$ with edge set
\begin{equation} \label{near-contraction}
E(T\odot e) = E(T) \sm \big\{wx \colon x\notin\{v,w\}\big\} \cup \big\{vx \colon x\notin\{v,w\}\big\}.
\end{equation}
Equivalently, contract the edge $e$, retaining the name $v$ for the resulting vertex, and introduce a new edge $e'=vw$, so that $w$ is a leaf.  See Figure~\ref{fig:near-contraction} for an example. There is a natural bijection between edges of $T$ and edges of $T\odot e$.
\begin{figure}[th]
\begin{center}
\begin{tikzpicture}
\draw(0,0)--(4,0);
\draw[red,thick](1,0)--(2,0);
\foreach \x in {0,1,2,3,4} \draw[fill=black] (\x,0) circle (.08);
\foreach \x/\s in {0/0, 1/1, 2.7/3, 3/3, 3.3/3, 3.8/4, 4.2/4} { \draw (\x,-1)--(\s,0);  \draw[fill=black] (\x,-1) circle (.08); }
\node at (2,-2) {$T$};
\node[red] at (1.5,-.25) {$e$}; \node at (1,.3) {$v$}; \node at (2,.3) {$w$};
\begin{scope}[shift={(8,0)}]
\draw(0,0)--(3,0);
\foreach \x/\s in {0/0, .8/1,  1.7/2, 2.0/2, 2.3/2, 2.8/3, 3.2/3} { \draw (\x,-1)--(\s,0);  \draw[fill=black] (\x,-1) circle (.08); }
\draw[red,thick] (1.2,-1)--(1,0);
\draw[fill=black] (1.2,-1) circle (.08); 
\foreach \x in {0,1,2,3} \draw[fill=black] (\x,0) circle (.08);
\node at (1.5,-2) {$T\odot e$};
\node at (1,.3) {$v$}; \node at (1.2,-1.3) {$w$}; \node[red] at (1.35,-.5) {$e'$};
\end{scope}
\end{tikzpicture}
\end{center}
\caption{Near-contraction of a non-leaf edge $e$ in a tree $T$.\label{fig:near-contraction}}
\end{figure}
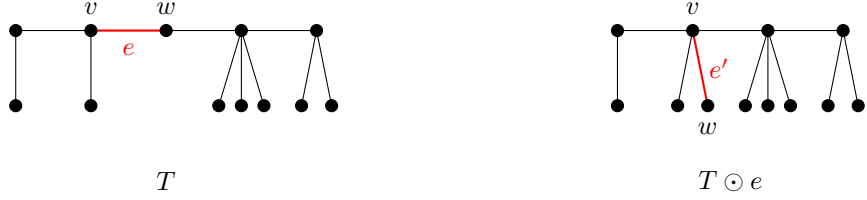

\begin{proposition} \label{prop:uhdp-recurrence}
Let $T_1,T_2$ be trees, let $v,w$ be vertices of $T_1$ and $T_2$, respectively, and let $e$ be the edge $vw$.
Let $T'_1=T_1\cup\{e\}$, $T'_2=T_2\cup\{e\}$, and $T=T_1\cup T_2\cup\{e\}$.  Then
\begin{equation*}
\uhdp{T} = \frac{y}{y+z} \left( \uhdp{T'_1} + \uhdp{T'_2} \right) + \frac{z}{y+z} \uhdp{T\odot e}.
\end{equation*}
\end{proposition}
Before giving the proof we need to set up notation. Given any vertex $v$ in $T$, let
\begin{align*}
\SSS'_{+v}(T) &= \{S\in\SSS'(T) \colon v\in S\}, &
\SSS'_{-v}(T) &= \{S\in\SSS'(T) \colon v\notin S\}.
\end{align*}
In addition, if $v,w$ are two distinct vertices in $T$, then define
\[\SSS'_{\pm v,\pm w}(T) = \SSS'_{\pm v}(T) \cap \SSS'_{\pm w}(T).\]
Using this notation we also define
\begin{align*}
\uhdpsup{\pm v}{T} &= \sum_{S\in\SSS'_{\pm v}(T)}y^{d(S)}z^{e(S)},
& \uhdpsup{\pm v,\pm w}{T} &= \sum_{S\in\SSS'_{\pm v,\pm w}(T)}y^{d(S)}z^{e(S)}.
\end{align*}
\begin{proof}[Proof of Proposition~\ref{prop:uhdp-recurrence}]
We start with the identity
\begin{equation} \label{eq:four_term1}
\uhdp{T} = \uhdpsup{+v,+w}{T}+\uhdpsup{+v,-w}{T}+\uhdpsup{-v,+w}{T}+\uhdpsup{-v,-w}{T},
\end{equation}
which is immediate from the definitions.

First, there are natural bijections $\SSS'_{+v,+w}(T)\to\SSS'_{+v,+w}(T\odot e)$ and $\SSS'_{+v,+w}(T\odot e)\to\SSS'_{+v,-w}(T\odot e)$, since $w$ is a leaf in $T\odot e$.  Thus, 
\[\uhdpsup{+v,+w}{T}=\uhdpsup{+v,+w}{T\odot e}=\frac{z}{y}\:\uhdpsup{+v,-w}{T\odot e}.\]
So 
\[\uhdpsup{+v}{T\odot e} = \uhdpsup{+v,+w}{T\odot e}+\uhdpsup{+v,-w}{T\odot e} = \left(1+\frac{y}{z}\right)\uhdpsup{+v,+w}{T\odot e}.\]
Combining these two equations yields
\begin{equation}
\label{eq:++}
\uhdpsup{+v,+w}{T}=\frac{z}{y+z}\uhdpsup{+v}{T\odot e}.
\end{equation}

Second, there is a bijection from $\SSS'_{-v,-w}(T)$ to $\SSS'_{-v}(T\odot e)$ since a subtree of $T\odot e$ that does not contain $v$ cannot contain the leaf $w$ (which is adjacent to $v$), so that
\begin{align}
\uhdpsup{-v,-w}{T}
&=\frac{y}{y+z}\uhdpsup{-v,-w}{T}+\frac{z}{y+z}\uhdpsup{-v,-w}{T} \notag\\
&=\frac{y}{y+z}\uhdpsup{-v,-w}{T}+\frac{z}{y+z}\uhdpsup{-v}{T\odot e}\label{eq:--}
\end{align}
where the first step is algebra and the second step uses the bijection.

Combining \eqref{eq:four_term1}, \eqref{eq:++} and \eqref{eq:--} yields
\begin{align}
\uhdp{T}&= \frac{z}{y+z}\uhdpsup{+v}{T\odot e}+\uhdpsup{+v,-w}{T}+\uhdpsup{-v,+w}{T}+\frac{y}{y+z}\uhdpsup{-v,-w}{T}+\frac{z}{y+z}\uhdpsup{-v}{T\odot e}\notag\\
&= \uhdpsup{+v,-w}{T}+\uhdpsup{-v,+w}{T}+\frac{y}{y+z}\uhdpsup{-v,-w}{T}+\frac{z}{y+z}\uhdp{T\odot e}.\label{eq:almost-there}
\end{align}
It remains to show that
\begin{equation} \label{eq:last-step}
\uhdpsup{+v,-w}{T}+\uhdpsup{-v,+w}{T}+\frac{y}{y+z}\uhdpsup{-v,-w}{T}
=\frac{y}{y+z} \left( \uhdp{T_1'} + \uhdp{T_2'} \right)
\end{equation}
which when combined with~\eqref{eq:almost-there} will give the desired result.

First, there is a bijection $\phi:\SSS'_{+v,-w}(T)\to\SSS'_{+v}(T_1)$; observe that $\phi(S)$ has one fewer external edge than $S$. Thus, 
\[\uhdpsup{+v,-w}{T}=y\uhdpsup{+v}{T_1}=\frac{y}{y+z}\uhdpsup{+v}{T_1'},\]
and similarly,
\[\uhdpsup{-v,+w}{T}=y\uhdpsup{+w}{T_2}=\frac{y}{y+z}\uhdpsup{+w}{T_2'}.\]
Finally, each subtree of $T$ containing neither $v$ nor $w$ is either a subtree of $T_1'$ that does not contain $v$, or a subtree of $T_2'$ that does not contain $w$. Hence,
\[\uhdpsup{-v,-w}{T} =\uhdpsup{-v}{T_1'}+\uhdpsup{-w}{T_2'}.\]
Combining the last three equations yields~\eqref{eq:last-step}, completing the proof.
\end{proof}
Proposition~\ref{prop:uhdp-recurrence} can be used to show that two non-isomorphic trees have the same (modified) HDP.
See Figure~\ref{fig:recurrence} for an example. We will exploit this idea further in the next section.
\begin{figure}[th]
% FIRST LINE
\begin{align*}
\uhdpop\left(
\begin{tikzpicture}[scale=0.7,baseline=-12pt]
\draw(0,0)--(4,0);
\draw[red,thick](1,0)--(2,0);
\foreach \x in {0,1,2,3,4} \draw[fill=black] (\x,0) circle (.08);
\foreach \x/\s in {0/0, 1/1, 2.8/3, 3.2/3, 3.8/4, 4.2/4} { \draw (\x,-1)--(\s,0);  \draw[fill=black] (\x,-1) circle (.08); }
\end{tikzpicture}\right)
&=
\frac{y}{y+z}\uhdpop\left(
\begin{tikzpicture}[scale=0.7,baseline=-12pt]
\draw(0,0)--(2,0);
\draw[red,thick](1,0)--(2,0);
\foreach \x in {0,1,2} \draw[fill=black] (\x,0) circle (.08);
\foreach \x/\s in {0/0, 1/1} { \draw (\x,-1)--(\s,0);  \draw[fill=black] (\x,-1) circle (.08); }
\end{tikzpicture}\right)+
\frac{y}{y+z}\uhdpop\left(
\begin{tikzpicture}[scale=0.7,baseline=-12pt]
\draw(1,0)--(4,0);
\draw[red,thick](1,0)--(2,0);
\foreach \x in {1,2,3,4} \draw[fill=black] (\x,0) circle (.08);
\foreach \x/\s in {2.8/3, 3.2/3, 3.8/4, 4.2/4} { \draw (\x,-1)--(\s,0);  \draw[fill=black] (\x,-1) circle (.08); }
\end{tikzpicture}\right)\\
&+
\frac{z}{y+z}\uhdpop\left(
\begin{tikzpicture}[scale=0.7,baseline=-12pt]
\draw(0,0)--(3,0);
\foreach \x/\s in {0/0, .8/1,  1.8/2, 2.2/2, 2.8/3, 3.2/3} { \draw (\x,-1)--(\s,0);  \draw[fill=black] (\x,-1) circle (.08); }
\draw[red,thick] (1.2,-1)--(1,0);
\draw[fill=black] (1.2,-1) circle (.08); 
\foreach \x in {0,1,2,3} \draw[fill=black] (\x,0) circle (.08);
\end{tikzpicture}
\right)
\end{align*}
% SECOND LINE
\begin{align*}
\uhdpop\left(
\begin{tikzpicture}[scale=0.7,baseline=-12pt]
\draw(0,0)--(4,0);
\draw[red,thick](2,0)--(3,0);
\foreach \x in {0,1,2,3,4} \draw[fill=black] (\x,0) circle (.08);
\foreach \x/\s in {0/0, 0.8/1, 1.2/1, 2/2, 3.8/4, 4.2/4} { \draw (\x,-1)--(\s,0);  \draw[fill=black] (\x,-1) circle (.08); }
\end{tikzpicture}\right)
&=
\frac{y}{y+z}\uhdpop\left(
\begin{tikzpicture}[scale=0.7,baseline=-12pt]
\draw(0,0)--(2,0);
\draw[red,thick](2,0)--(3,0);
\foreach \x in {0,1,2,3} \draw[fill=black] (\x,0) circle (.08);
\foreach \x/\s in {0/0, 0.8/1, 1.2/1, 2/2} { \draw (\x,-1)--(\s,0);  \draw[fill=black] (\x,-1) circle (.08); }
\end{tikzpicture}\right)+
\frac{y}{y+z}\uhdpop\left(
\begin{tikzpicture}[scale=0.7,baseline=-12pt]
\draw(1,0)--(3,0);
\draw[red,thick](1,0)--(2,0);
\foreach \x in {1,2,3} \draw[fill=black] (\x,0) circle (.08);
\foreach \x/\s in {2.8/3, 3.2/3} { \draw (\x,-1)--(\s,0);  \draw[fill=black] (\x,-1) circle (.08); }
\end{tikzpicture}\right)\\
&+
\frac{z}{y+z}\uhdpop\left(
\begin{tikzpicture}[scale=0.7,baseline=-12pt]
\draw(0,0)--(3,0);
\foreach \x/\s in {0/0, .8/1,  1.8/2, 1.2/1, 2.8/3, 3.2/3} { \draw (\x,-1)--(\s,0);  \draw[fill=black] (\x,-1) circle (.08); }
\draw[red,thick] (2.2,-1)--(2,0);
\draw[fill=black] (2.2,-1) circle (.08); 
\foreach \x in {0,1,2,3} \draw[fill=black] (\x,0) circle (.08);
\end{tikzpicture}
\right)
\end{align*}
\caption{Application of Proposition \ref{prop:uhdp-recurrence}. Note that the right-hand sides of both equations are equal, proving that the two trees on the left-hand sides have the same HDP.
\label{fig:recurrence}}
\end{figure}

As a remark, we have not been able to obtain recurrences for $\gdp{T}$ similar to those for $\hdp{T}$.

\section{Families of trees with the same half-generalized degree polynomial} \label{sec:same-HDP}
In this section we will construct arbitrarily large families of non-isomorphic trees with the same half-generalized degree polynomial by exploiting the recurrence given by Proposition \ref{prop:uhdp-recurrence}.

Recall that a \defterm{composition} of an integer $n$ is an ordered list of positive integers $\alpha=(a_1,\dots,a_k)$ that add up to $n$.   The \defterm{length} of~$\alpha$ is $\ell(\alpha)=k$. The \defterm{reverse}  of~$\alpha$ is $\alpha^*:=(\alpha_k,\alpha_{k-1},\dots,a_1)$. If $\alpha=(a_1,\dots,a_k)$ and $\beta=(b_1,\dots,b_m)$ with $|\alpha|=|\beta|$, we say that $\beta$ is a \defterm{coarsening} of~$\alpha$, written $\beta\geq\alpha$, if every partial sum of $\beta$ (i.e., every number $b_1+\cdots+b_j$ for some $j$) is also a partial sum of $\alpha$.  Coarsening is a partial order on compositions of $n$.
The \defterm{concatenation} of $\alpha$ and $\beta$ is
\[\alpha\cdot\beta = (a_1,\dots, a_k,b_1,\dots,b_m),\]
and the \defterm{near-concatenation} is
\[\alpha\odot\beta = (a_1,\dots, a_{k-1},a_k+b_1,b_2,\dots,b_m).\]

A \defterm{caterpillar} is a tree with the property that deleting all its leaves produces a path $\spine(T)$, called the \defterm{spine} of the caterpillar.
We write $\Cat(a_1,\dots,a_k)$ for the caterpillar with $k$ spine vertices $v_1,\dots,v_k$, adjacent to $a_1-1,\dots,a_k-1$ leaves respectively, so that the total number of vertices is $n=a_1+\cdots+a_k$.  (See Figure~\ref{fig:cats} for an example.)
The composition $\alpha=(a_1,\dots,a_k)$ is called the \defterm{signature} of the caterpillar.
The signature is well-defined, and a complete invariant, up to reversal.  Moreover, $a_1,a_k\geq 2$ since each end vertex of the spine is adjacent to at least one leaf; the other numbers $a_i$ are unconstrained. Let $\mathcal{C}$ denote the set of compositions where the first and last parts are larger than $1$, i.e., compositions that are signatures of caterpillars.
\begin{figure}[!th]
\begin{center}
\begin{tikzpicture}
\foreach \x in {0,...,6} \draw[fill=black] (\x,0) circle (.08);
\draw(0,0)--(6,0);
\foreach \x/\s in {0/0, 1/1, 2.7/3, 3/3, 3.3/3, 3.8/4, 4.2/4, 5.5/6, 5.75/6, 6/6, 6.25/6, 6.5/6} { \draw (\x,-1)--(\s,0);  \draw[fill=black] (\x,-1) circle (.08); }
\end{tikzpicture}
\end{center}
\caption{The caterpillar $\Cat(2,2,1,4,3,1,6)$.\label{fig:cats}}
\end{figure}
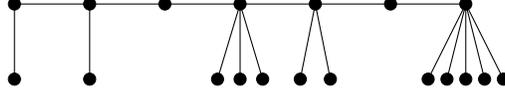

Observe that deleting the edge $e=v_k v_{k+1}$ from $\Cat(\alpha\cdot\beta)$ produces the forest with components $\Cat(\alpha)$ and $\Cat(\beta)$.  Moreover, $\Cat(\alpha\odot\beta)$ is just the near-contraction $\Cat(\alpha\cdot\beta)\odot e$ (see~\eqref{near-contraction}).

For convenience, given a composition $\alpha$, define
\[\Hone(\alpha)=\uhdp{\Cat(1\odot \alpha\odot 1)}.\]
In this notation, Proposition~\ref{prop:uhdp-recurrence} becomes
\begin{equation}
\label{eq:recurrence_h}
\Hone(\alpha\cdot\beta) = \frac{y}{y+z}(\Hone(\alpha)+\Hone(\beta))
+\frac{z}{y+z} \Hone(\alpha\odot\beta).
\end{equation}
Define
\begin{equation} \label{define-Htwo}
\Htwo(\alpha)=
\Htwo(\alpha)(y,z,x_1,x_2,\dots)=
\sum_{\gamma\geq \alpha} \frac{{y}^{\ell(\gamma)-1}z^{\ell(\alpha)-\ell(\gamma)}}{(y+z)^{\ell(\alpha)-1}}\sum_{i=1}^{\ell(\gamma)}x_{\gamma_i}.
\end{equation}
Our next step is to show that the power series $\Htwo(\alpha)$ satisfies the same recurrence as $\Hone(\alpha)$, which will enable us to obtain a closed form for the HDP of a caterpillar using the right-hand side of~\eqref{define-Htwo}.

\begin{proposition} \label{prop:Htwo-recurrence}
For all compositions $\alpha$ and $\beta$,
\begin{equation} \label{eq:recurrence-H}
\Htwo(\alpha\cdot\beta)=\frac{y}{y+z}\left(\Htwo(\alpha)+\Htwo(\beta)\right)+\frac{z}{y+z}\Htwo(\alpha\odot\beta).
\end{equation}
\end{proposition}

\begin{proof}
In each coarsening $\gamma\geq\alpha\cdot\beta$, the last part of $\alpha$ and the first part of $\beta$ are either merged or kept separate.  In the first case, $\gamma$ coarsens $\alpha\odot\beta$.  In the second, $\gamma$ is the concatenation of a coarsening of $\alpha$ with a coarsening of $\beta$.  Therefore, we may split up the expression for $\Htwo(\alpha\cdot\beta)$ given by~\eqref{define-Htwo} as
\begin{equation} \label{eq:S}
\Htwo(\alpha\cdot \beta)=
\underbrace{
\sum_{\gamma\geq\alpha\odot\beta}
\frac{y^{\ell(\gamma)-1}z^{\ell(\alpha\cdot\beta)-\ell(\gamma)}}{(y+z)^{\ell(\alpha\cdot\beta)-1}} 
\sum_{i=1}^{\ell(\gamma)} x_{\gamma_i}}_A
~+~
\underbrace{\sum_{\gamma'\geq\alpha}\sum_{\gamma''\geq \beta}
\frac{y^{\ell(\gamma'\cdot\gamma'')-1}z^{\ell(\alpha\cdot\beta)-\ell(\gamma'\cdot\gamma'')}}{(y+z)^{\ell(\alpha\cdot\beta)-1}} 
\sum_{i=1}^{\ell(\gamma'\cdot\gamma'')} x_{(\gamma'\cdot\gamma'')_i}}_B.
\end{equation}
First, observe that
\begin{align}
A&=\frac{z}{y+z}\sum_{\gamma\geq\alpha\odot\beta} \frac{y^{\ell(\gamma)-1}z^{\ell(\alpha\odot\beta)-\ell(\gamma)}}{(y+z)^{\ell(\alpha\odot\beta)-1}} 
\sum_{i=1}^{\ell(\gamma)} x_{\gamma_i}=\frac{z}{y+z}\Htwo(\alpha\odot\beta).\label{eq:A}
\end{align}

Second,
\begin{align}
B&= 
\frac{y}{y+z}\sum_{\gamma'\geq\alpha}\sum_{\gamma''\geq \beta}
\frac{y^{\ell(\gamma'\cdot\gamma'')-2}z^{\ell(\alpha\cdot\beta)-\ell(\gamma'\cdot\gamma'')}}{(y+z)^{\ell(\alpha\cdot\beta)-2}} 
\sum_{i=1}^{\ell(\gamma'\cdot\gamma'')} x_{(\gamma'\cdot\gamma'')_i}
\notag\\
&=\frac{y}{y+z}\sum_{\gamma'\geq\alpha}
\sum_{\gamma''\geq \beta}
\frac{y^{\ell(\gamma')-1}z^{\ell(\alpha)-\ell(\gamma')}}
{(y+z)^{\ell(\alpha)-1}}
\frac{y^{\ell(\gamma'')-1}z^{\ell(\beta)-\ell(\gamma'')}}{(y+z)^{\ell(\beta)-1}} 
\left(\sum_{i=1}^{\ell(\gamma')} x_{\gamma'_i}+\sum_{i=1}^{\ell(\gamma'')} x_{\gamma''_i}\right)
\notag\\
&= 
\frac{y}{y+z} \left(
\left(\sum_{\gamma''\geq \beta}
\frac{y^{\ell(\gamma'')-1}z^{\ell(\beta)-\ell(\gamma'')}}{(y+z)^{\ell(\beta)-1}}
\right)
\Htwo(\alpha)
+
\left(
\sum_{\gamma'\geq \alpha}
\frac{y^{\ell(\gamma')-1}z^{\ell(\alpha)-\ell(\gamma')}}
{(y+z)^{\ell(\alpha)-1}}\right)
\Htwo(\beta)
\right).\label{eq:B}
\end{align}
On the other hand,
\begin{equation} \label{eq:numer}
\sum_{\gamma'\geq\alpha} y^{\ell(\gamma')-1}z^{\ell(\alpha)-\ell(\gamma')} = \sum_{k=1}^{\ell(\alpha)} \binom{\ell(\alpha)-1}{k-1} y^{k-1}z^{\ell(\alpha)-k}= (y+z)^{\ell(\alpha)-1},
\end{equation}
so the parenthesized sums in~\eqref{eq:B} may be dropped, and then substituting~\eqref{eq:B} and~\eqref{eq:A} into~\eqref{eq:S} yields the desired equality.
\end{proof}

\begin{proposition} \label{prop:Hbar-from-Htwo}
For every composition $\alpha$, we have
\begin{align}
\uhdp{\Cat(1\odot\alpha\odot1)}
&= \Htwo(\alpha)(y,z,x_1,x_2,\dots)|_{x_1=(y+z)^{2},\ x_2=(y+z)^{3},\ \ldots,\ x_i=(y+z)^{i+1},\ \ldots} \notag\\
&= \sum_{\gamma\geq\alpha} \frac{y^{\ell(\gamma)-1}z^{\ell(\alpha)-\ell(\gamma)}}{(y+z)^{\ell(\alpha)-1}} 
\sum_{i=1}^{\ell(\gamma)} (y+z)^{\gamma_i+1}.
\label{hdp-closed}
\end{align}
\end{proposition}

\begin{proof}
We induct on $\ell(\alpha)$. For the base case, say $\alpha=(a)$.  Then $1\odot\alpha\odot1=(a+2)$; that is, $\Cat(a+2)$ is the star with $a+1$ edges.  So
\[\Hone(\alpha) =\uhdp{\Cat(a+2)}=\frac{1}{(y+z)^{-1}} (y+z)^a = (y+z)^{a+1},\]
which is the right-hand side of~\eqref{hdp-closed}.
Meanwhile, the inductive step follows directly from \eqref{eq:recurrence_h}, \eqref{eq:recurrence-H} and the induction hypothesis.
\end{proof} 

The preceding machinery can be used to construct caterpillars with the same HDP. The procedure follows the construction of compositions with the same ribbon Schur function \cite{BTvW} or $\mathcal{L}$-polynomial \cite{APZ}. Given two compositions $\alpha=(a_1,\ldots,a_k)$ and $\beta$, define 
\[\alpha\circ \beta = \beta^{\odot a_1}\cdot \beta^{\odot a_2}\cdots\beta^{\odot a_k}\]
where $\beta^{\odot i}=\beta\odot\cdots\odot\beta$ ($i$ times).
This operation satisfies the identities
\begin{align}
(\alpha\cdot\gamma)\circ\beta &= (\alpha\circ\beta)\cdot(\gamma\circ\beta),\label{circ:1}\\
(\alpha\odot\gamma)\circ\beta &= (\alpha\circ\beta)\odot(\gamma\circ\beta),\label{circ:2}\\
(\alpha\circ\beta)^* &= \alpha^*\circ\beta^*.\label{circ:reverse}
\end{align}

\begin{proposition}
\label{prop-composition-caterpillars}
For all compositions $\alpha$ and $\beta$, we have 
\[\Htwo(\alpha\circ\beta) = \Htwo(\alpha)\big\vert_{x_i=\Htwo(\beta^{\odot i})}.\]
\end{proposition}

\begin{proof}
We induct on $k=\ell(\alpha)$. For the base case, suppose $\alpha=(a)$, so by definition
$\Htwo(\alpha\circ\beta)=\Htwo(\beta^{\odot a})$ and $\Htwo(a)=x_a$. Thus $\Htwo(a)|_{x_i=\Htwo(\beta^{\odot i})} = \Htwo(\beta^{\odot a})$ as desired.  For the inductive step, let $\alpha'=(a_1,\dots,a_{k-1})$, so that
\begin{align*}
\Htwo(\alpha\circ\beta)
&=\Htwo((\alpha'\circ\beta)\cdot (a_k\circ\beta))\\
&=\frac{y}{y+z}\left(\Htwo(\alpha'\circ\beta)+
\Htwo(a_k\circ\beta)\right)+
\frac{z}{y+z}\Htwo((\alpha'\odot a_k)\circ\beta)\\
&=\frac{y}{y+z}\left(\Htwo(\alpha')+
\Htwo( a_k)\right)|_{x_i=\Htwo(\beta^{\odot i})}+
\frac{z}{y+z}\Htwo(\alpha'\odot a_k)|_{x_i=\Htwo(\beta^{\odot i})}\\
&=\Htwo(\alpha)|_{x_i=\Htwo(\beta^{\odot i})}.\qedhere
\end{align*}
\end{proof}

A \defterm{factorization} of a composition $\alpha$ is an equality
$\alpha=\alpha_1\circ\alpha_2\circ\cdots\circ\alpha_k$. 
The factorization is \defterm{nontrivial} if (i) no $\alpha_i$ is the composition 1; (ii) no two consecutive factors both have length 1; and (iii)
no two consecutive factors both have all parts equal to 1.
The factorization is \defterm{irreducible} if it is nontrivial and no $\alpha_i$ admits a nontrivial factorization.  In fact every composition admits a unique irreducible factorization~\cite[Theorem 3.6]{BTvW}.

Let $\alpha$ be a composition with irreducible factorization $\alpha_1\circ\cdots\circ\alpha_k$.  A \defterm{switch} of $\alpha$ is a composition $\beta$ of the form $\beta_1\circ\cdots\circ\beta_k$, where $\beta_i\in\{\alpha_i,\alpha_i^*\}$ for each $i$.  Unique factorization implies that switching is an equivalence relation on compositions.

\begin{theorem} \label{thm:equiv-class}
\begin{enumerate}
\item Suppose that $\alpha$ and $\beta$ are related by switching.  Then $\Htwo(\alpha)=\Htwo(\beta)$, and consequently
\[\hdp{\Cat(1\odot\alpha\odot 1)} = \hdp{\Cat(1\odot\beta\odot 1)}.\]
\item In particular, if $\alpha$ has $q$ irreducible factors that are not palindromes, then the $\Htwo$-equivalence class of $\alpha$ contains at least $2^{q-1}$ non-equivalent compositions.  In this case, $\Cat(1\odot\alpha\odot1)$ is one of at least $2^{q-1}$ pairwise non-isomorphic caterpillars with the same half-generalized degree polynomial (equivalently, by Theorem~\ref{thm:sub-hdp}, with the same subtree polynomial).
\end{enumerate}
In particular, there exist arbitrarily large sets of non-isomorphic caterpillars with the same HDP and STP.
\end{theorem}

\begin{proof}
First, observe that $\Htwo(\alpha)=\Htwo(\alpha^*)$, because $\gamma\geq\alpha$ if and only if $\gamma^*\geq\alpha^*$, and $\gamma$ and $\gamma^*$ make the same contribution to the right-hand side of~\eqref{define-Htwo}.  Second, observe that
\[\Htwo(\alpha^*\circ\beta^*)=\Htwo(\alpha\circ\beta)=\Htwo(\alpha^*\circ\beta)=\Htwo(\alpha\circ\beta^*)\]
The first and third equalities follow from~\eqref{circ:reverse}, while the second follows from the first identity of Proposition~\ref{prop-composition-caterpillars}.
The desired corollary now follows by induction.
\end{proof}

\begin{example} \label{exa:same-hdp}
Consider the irreducible factorizations
\begin{align*}
\alpha &=(1,2)\circ(1,2) = (1,2)^{\odot1} \cdot (1,2)^{\odot2} = (1,2)\cdot(1,3,2) = (1,2,1,3,2),\\
\beta &= (2,1)\circ(1,2) = (1,2)^{\odot2} \cdot (1,2)^{\odot1} = (1,3,2)\cdot(1,2) = (1,3,2,1,2).
\end{align*}
Then $\Htwo(\alpha)=\Htwo(\beta)$ by Theorem~\ref{thm:equiv-class}, and the corresponding caterpillars
\[\Cat(1\odot\alpha\odot1)=\Cat(2,2,1,3,3), \qquad \Cat(1\odot\beta\odot1)=\Cat(2,3,2,1,3)\]
have the same half-generalized and subtree polynomials.  In fact, these are the two smallest such trees, shown above in Figure~\ref{fig:sameGDP}.
\end{example}

As a consequence, we prove the conjecture of Eisenstat and Gordon \cite[Conjecture~2.8]{EG}, which we now explain.
Let $p(x)$ be a polynomial whose coefficients are all 0's and 1's, with no two consecutive 0's; we may as well assume that the leading and trailing coefficients are 1.  Call such a polynomial \defterm{gap-free}.  Let $a<b$ be positive integers.
Expand the polynomials $(a+bx)p(x)$ and $(b+ax)p(x)$, read off the coefficient lists, and add 1 to the first and last terms to obtain lists $L_{p,1},L_{p,2}$, which we regard as the signatures of caterpillars $C_{p,1},C_{p,2}$.  For example, if $p(x)=1+x+x^3$ then
\begin{align*}
(a+bx)p(x) &= a+(a+b)x+bx^2+ax^3+bx^4 & L_{p,1}&=(a+1,a+b,b,a,b+1) & C_{p,1}&=\Cat(L_{p,1}),\\
(b+ax)p(x) &= b+(a+b)x+ax^2+bx^3+ax^4 & L_{p,2}&=(b+1,a+b,a,b,a+1) & C_{p,2}&=\Cat(L_{p,2}).
\end{align*}
For example, when $(a,b)=(1,2)$, this construction produces the caterpillars in Figure~\ref{fig:sameGDP}.

\begin{corollary}[Eisenstat--Gordon Conjecture] 
For every gap-free polynomial $p(x)$ and any positive integers $a,b$, the caterpillars $C_{p,1}$ and $C_{p,2}$ have the same subtree polynomial.
\end{corollary}

\begin{proof}
This statement is a special case of Theorem~\ref{thm:equiv-class}, for the following reasons.
Write $0=i_0<i_1<i_2<\ldots<i_d=\deg{p}$, where $\{i_1,\ldots,i_{d-1}\}$ are all the indices of the coefficients of $p$ equal to $0$. Let $\beta=(i_1-i_0,i_2-i_1,\ldots,i_d-i_{d-1})$ and $\alpha=(a,b)$.  Then
\[C_{p,1}=\Cat(1\odot (\beta \circ \alpha)\odot 1)\quad\text{and}\quad 
C_{p,2}=\Cat(1\odot (\beta\circ \alpha^*)\odot 1),\]
so indeed $C_{p,1}$ and $C_{p,2}$ have the same HDP and thus also the same STP.
\end{proof}

\section{Further remarks}

\subsection{Beyond caterpillars}
The near-concatenation operation $\odot$, and thus the construction of Theorem~\ref{thm:equiv-class} may be extended to trees that are not caterpillars.  Say that a \defterm{polarized tree} is a tree $T$ with two distinguished vertices, the \defterm{left end} and \defterm{right end}.  Given two polarized trees $T,T'$ with left ends $L,L'$ and right ends $R,R'$ respectively,
the \defterm{concatenation} is the polarized tree $T\cdot T'=(T+T')\cup\{RL'\}$ (where $+$ denotes disjoint union), with left end $u$ and right end $R'$. The \defterm{near-concatenation} is $T\odot T'=(T\cdot T')\odot RL'$, with the same left and right ends.

Let $\beta=(\beta_1,\ldots,\beta_m)$ be a composition and $T$ a polarized tree.  Define 
\[ \beta\circ T = T^{\odot \beta_1} \cdot T^{\odot \beta_2}\cdot \ldots \cdot T^{\odot\beta_m}\]
where
\[
T^{\odot i} = \underbrace{T\odot T\odot\cdots\odot T}_{\text{$i$ times}}.
\]
Examples of these constructions are given in Figure~\ref{fig:polarized}.

\begin{figure}[ht]
\begin{center}
\begin{tikzpicture}
\newcommand{\dx}{0.7}
\newcommand{\dy}{0.5}
\draw (0,0)--(3*\dx,0)--(3*\dx,\dy) (\dx-0.25,\dy)--(\dx,0)--(\dx+.25,\dy) (2*\dx-.25,-2*\dy)--(2*\dx,-\dy)--(2*\dx+.25,-2*\dy) (2*\dx,-\dy)--(2*\dx,0);
\foreach \x/\y in {\dx/0, 2*\dx/0, \dx-.25/\dy, \dx+.25/\dy, 2*\dx/-\dy, 2*\dx-.25/-2*\dy, 2*\dx+.25/-2*\dy, 3*\dx/\dy} \draw[thick,fill=black] (\x,\y) circle (.075);
\draw[very thick, fill=white] (0,0) circle (.2); \node at (0,0) {\footnotesize\sf L};
\draw[very thick, fill=white] (3*\dx,0) circle (.2); \node at (3*\dx,0) {\footnotesize\sf R};
\node at (1.5*\dx,-4*\dy) {$T$};

\begin{scope}[shift={(9*\dx,0)}]
\draw (0,0)--(3*\dx,0)--(3*\dx,\dy) (\dx-0.25,\dy)--(\dx,0)--(\dx+.25,\dy) (2*\dx-.25,-2*\dy)--(2*\dx,-\dy)--(2*\dx+.25,-2*\dy) (2*\dx,-\dy)--(2*\dx,0);
\foreach \x/\y in {\dx/0, 2*\dx/0, \dx-.25/\dy, \dx+.25/\dy, 2*\dx/-\dy, 2*\dx-.25/-2*\dy, 2*\dx+.25/-2*\dy, 3*\dx/\dy} \draw[thick,fill=black] (\x,\y) circle (.075);
\draw[very thick, fill=white] (0,0) circle (.2); \node at (0,0) {\footnotesize\sf L};
\end{scope}
\begin{scope}[shift={(12*\dx,0)}]
	\draw (3*\dx,0)--(3*\dx,-\dy) (0,-\dy)--(0,0);
	\draw[very thick, fill=black] (0,-\dy) circle (.075);
	\draw[very thick, fill=black] (3*\dx,-\dy) circle (.075);
	\draw (0,0)--(3*\dx,0)--(3*\dx,\dy) (\dx-0.25,\dy)--(\dx,0)--(\dx+.25,\dy) (2*\dx-.25,-2*\dy)--(2*\dx,-\dy)--(2*\dx+.25,-2*\dy) (2*\dx,-\dy)--(2*\dx,0);
	\foreach \x/\y in {\dx/0, 2*\dx/0, \dx-.25/\dy, \dx+.25/\dy, 2*\dx/-\dy, 2*\dx-.25/-2*\dy, 2*\dx+.25/-2*\dy, 3*\dx/\dy} \draw[thick,fill=black] (\x,\y) circle (.075);
	\draw[very thick, fill=white] (0,0) circle (.15);
	\node at (1.5*\dx,-4*\dy) {$T^{\odot3}$};
\end{scope}
\begin{scope}[shift={(15*\dx,0)}]
	\draw (0,0)--(3*\dx,0)--(3*\dx,\dy) (\dx-0.25,\dy)--(\dx,0)--(\dx+.25,\dy) (2*\dx-.25,-2*\dy)--(2*\dx,-\dy)--(2*\dx+.25,-2*\dy) (2*\dx,-\dy)--(2*\dx,0);
	\foreach \x/\y in {\dx/0, 2*\dx/0, \dx-.25/\dy, \dx+.25/\dy, 2*\dx/-\dy, 2*\dx-.25/-2*\dy, 2*\dx+.25/-2*\dy, 3*\dx/\dy} \draw[thick,fill=black] (\x,\y) circle (.075);
	\draw[very thick, fill=white] (0,0) circle (.15);
	\draw[very thick, fill=white] (3*\dx,0) circle (.2); \node at (3*\dx,0) {\footnotesize\sf R};
\end{scope}

\begin{scope}[shift={(0,-8*\dy)}]
	\node at (10*\dx,-4*\dy) {$(3,1,2)\circ T$};
	\draw (0,0)--(3*\dx,0)--(3*\dx,\dy) (\dx-0.25,\dy)--(\dx,0)--(\dx+.25,\dy) (2*\dx-.25,-2*\dy)--(2*\dx,-\dy)--(2*\dx+.25,-2*\dy) (2*\dx,-\dy)--(2*\dx,0);
	\foreach \x/\y in {\dx/0, 2*\dx/0, \dx-.25/\dy, \dx+.25/\dy, 2*\dx/-\dy, 2*\dx-.25/-2*\dy, 2*\dx+.25/-2*\dy, 3*\dx/\dy} \draw[thick,fill=black] (\x,\y) circle (.075);
	\draw[very thick, fill=white] (0,0) circle (.2); \node at (0,0) {\footnotesize\sf L};
	\begin{scope}[shift={(3*\dx,0)}]
		\draw (3*\dx,0)--(3*\dx,-\dy) (0,-\dy)--(0,0);
		\draw[very thick, fill=black] (0,-\dy) circle (.075);
		\draw[very thick, fill=black] (3*\dx,-\dy) circle (.075);
		\draw (0,0)--(3*\dx,0)--(3*\dx,\dy) (\dx-0.25,\dy)--(\dx,0)--(\dx+.25,\dy) (2*\dx-.25,-2*\dy)--(2*\dx,-\dy)--(2*\dx+.25,-2*\dy) (2*\dx,-\dy)--(2*\dx,0);
		\foreach \x/\y in {\dx/0, 2*\dx/0, \dx-.25/\dy, \dx+.25/\dy, 2*\dx/-\dy, 2*\dx-.25/-2*\dy, 2*\dx+.25/-2*\dy, 3*\dx/\dy} \draw[thick,fill=black] (\x,\y) circle (.075);
		\draw[very thick, fill=white] (0,0) circle (.15);
	\end{scope}
	\begin{scope}[shift={(6*\dx,0)}]
		\draw (0,0)--(3*\dx,0)--(3*\dx,\dy) (\dx-0.25,\dy)--(\dx,0)--(\dx+.25,\dy) (2*\dx-.25,-2*\dy)--(2*\dx,-\dy)--(2*\dx+.25,-2*\dy) (2*\dx,-\dy)--(2*\dx,0);
		\foreach \x/\y in {\dx/0, 2*\dx/0, \dx-.25/\dy, \dx+.25/\dy, 2*\dx/-\dy, 2*\dx-.25/-2*\dy, 2*\dx+.25/-2*\dy, 3*\dx/\dy} \draw[thick,fill=black] (\x,\y) circle (.075);
		\draw[very thick, fill=white] (0,0) circle (.15);
		\draw (3*\dx,0)--(4*\dx,0); % connecting edge
		\draw[very thick, fill=white] (3*\dx,0) circle (.15);
	\end{scope}
	\begin{scope}[shift={(10*\dx,0)}]
		\draw (0,0)--(3*\dx,0)--(3*\dx,\dy) (\dx-0.25,\dy)--(\dx,0)--(\dx+.25,\dy) (2*\dx-.25,-2*\dy)--(2*\dx,-\dy)--(2*\dx+.25,-2*\dy) (2*\dx,-\dy)--(2*\dx,0);
		\foreach \x/\y in {\dx/0, 2*\dx/0, \dx-.25/\dy, \dx+.25/\dy, 2*\dx/-\dy, 2*\dx-.25/-2*\dy, 2*\dx+.25/-2*\dy, 3*\dx/\dy} \draw[thick,fill=black] (\x,\y) circle (.075);
		\draw[very thick, fill=white] (0,0) circle (.15);
		\draw (3*\dx,0)--(4*\dx,0); % connecting edge
		\draw[very thick, fill=white] (3*\dx,0) circle (.15);
	\end{scope}
	\begin{scope}[shift={(14*\dx,0)}]
		\draw (0,0)--(3*\dx,0)--(3*\dx,\dy) (\dx-0.25,\dy)--(\dx,0)--(\dx+.25,\dy) (2*\dx-.25,-2*\dy)--(2*\dx,-\dy)--(2*\dx+.25,-2*\dy) (2*\dx,-\dy)--(2*\dx,0);
		\foreach \x/\y in {\dx/0, 2*\dx/0, \dx-.25/\dy, \dx+.25/\dy, 2*\dx/-\dy, 2*\dx-.25/-2*\dy, 2*\dx+.25/-2*\dy, 3*\dx/\dy} \draw[thick,fill=black] (\x,\y) circle (.075);
		\draw[very thick, fill=white] (0,0) circle (.15);
	\end{scope}
	\begin{scope}[shift={(17*\dx,0)}]
		\draw (0,-\dy)--(0,0);
		\draw[very thick, fill=black] (0,-\dy) circle (.075);
		\draw (0,0)--(3*\dx,0)--(3*\dx,\dy) (\dx-0.25,\dy)--(\dx,0)--(\dx+.25,\dy) (2*\dx-.25,-2*\dy)--(2*\dx,-\dy)--(2*\dx+.25,-2*\dy) (2*\dx,-\dy)--(2*\dx,0);
		\foreach \x/\y in {\dx/0, 2*\dx/0, \dx-.25/\dy, \dx+.25/\dy, 2*\dx/-\dy, 2*\dx-.25/-2*\dy, 2*\dx+.25/-2*\dy, 3*\dx/\dy} \draw[thick,fill=black] (\x,\y) circle (.075);
		\draw[very thick, fill=white] (0,0) circle (.15);
		\draw[very thick, fill=white] (3*\dx,0) circle (.2); \node at (3*\dx,0) {\footnotesize\sf R};
	\end{scope}
\end{scope}
\end{tikzpicture}
\end{center}
\caption{Operations on polarized trees.}
\label{fig:polarized}
\end{figure}
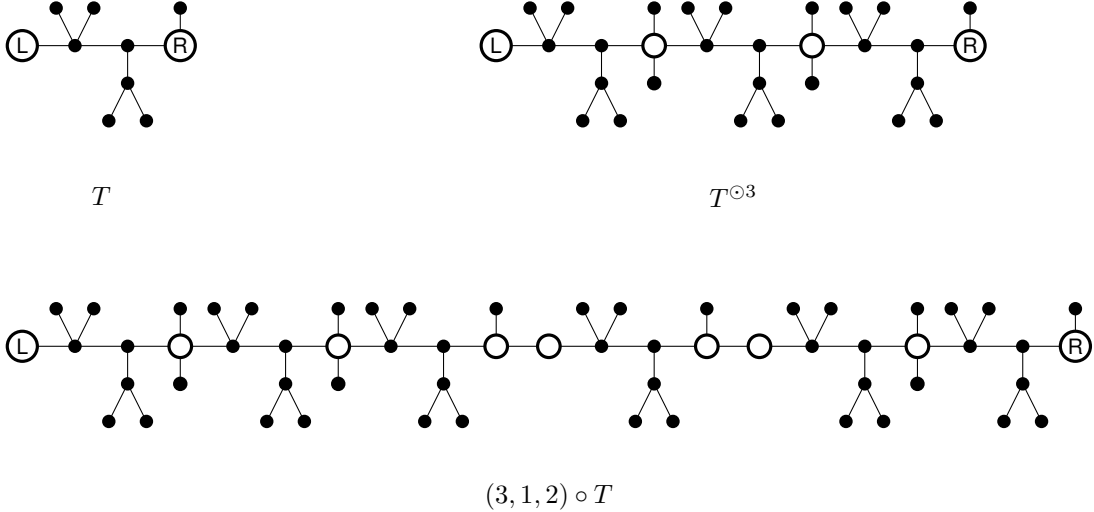

By extending the arguments of Section~\ref{sec:same-HDP} from caterpillars to polarized trees, we can generalize
Proposition~\ref{prop-composition-caterpillars} and Theorem~\ref{thm:equiv-class} to arbitrary trees, as follows.
\begin{theorem}
\label{thm:h-class-general}
Let $T$ be a tree, and $\alpha$ a composition. Then 
\[\uhdp{1\odot (\alpha\circ T)\odot 1} = \Htwo(\alpha)|_{x_i=\uhdp{1\odot T^{\odot i}\odot 1}}.\]
Moreover, if $\alpha$ and $\beta$ are related by switching, then
\[\hdp{\Cat(1\odot (\alpha\circ T)\odot 1)} = \hdp{\Cat(1\odot (\beta\circ T)\odot 1)}.\]
\end{theorem}

Not all pairs of trees with the same HDP arise from the construction of Theorem~\ref{thm:h-class-general}.  The unique smallest example, obtained by computer experimentation, is shown in Figure \ref{fig:counter}.

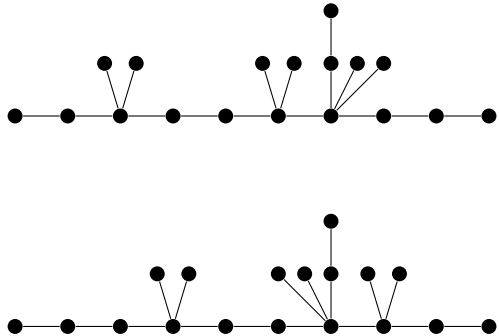
\begin{figure}
\begin{center}
\begin{tikzpicture}[scale=.7,every node/.style={circle,fill=black,inner sep = 2pt}]
\node (a0) at (5,0) {};
\node (a8) at (6,0) {};
\node (a16) at (5.3,1) {};
\node (a17) at (4.7,1) {};
\node (a1) at (4,0) {};
\node (a2) at (3,0) {};
\node (a3) at (2,0) {};
\node (a4) at (1,0) {};
\node (a5) at (0,0) {};
\node (a6) at (1.7,1) {};
\node (a7) at (2.3,1) {};
\node (a9) at (7,0) {};
\node (a10) at (8,0) {};
\node (a11) at (9,0) {};
\node (a12) at (6,1) {};
\node (a13) at (6,2) {};
\node (a14) at (6.5,1) {};
\node (a15) at (7,1) {};
\draw (a0)--(a1)  (a0)--(a8)  (a0)--(a16)  (a0)--(a17)  (a1)--(a2)  (a2)--(a3)  (a3)--(a4)--(a5)  (a6)--(a3)--(a7)  (a13)--(a12)--(a8)--(a9)--(a10)--(a11)  (a14)--(a8)--(a15);
\begin{scope}[shift={(0,-4)}]
	\node (a0) at (5,0) {};
	\node (a8) at (6,0) {};
	\node (a16) at (5,1) {};
	\node (a17) at (5.5,1) {};
	\node (a1) at (4,0) {};
	\node (a2) at (3,0) {};
	\node (a3) at (2,0) {};
	\node (a4) at (1,0) {};
	\node (a5) at (0,0) {};
	\node (a6) at (2.7,1) {};
	\node (a7) at (3.3,1) {};
	\node (a9) at (7,0) {};
	\node (a10) at (8,0) {};
	\node (a11) at (9,0) {};
	\node (a12) at (6.7,1) {};
	\node (a13) at (7.3,1) {};
	\node (a14) at (6,1) {};
	\node (a15) at (6,2) {};
	\draw (a0)--(a8)  (a0)--(a1)--(a2)--(a3)--(a4)--(a5)  (a7)--(a2)--(a6)  (a14)--(a8)--(a9)  (a17)--(a8)--(a16)    (a12)--(a9)--(a10)--(a11)  (a9)--(a13)  (a14)--(a15);
\end{scope}
\end{tikzpicture}
\end{center}
\caption{Trees with the same HDP that cannot be obtained as $\alpha\circ T$}
\label{fig:counter}
\end{figure}

\subsection{Comparing the GDP and the HDP} \label{gdp-and-hdp}

For trees $T$ and $T'$, it is clear that $\gdp{T}=\gdp{T'}$ implies $\hdp{T}=\hdp{T'}$.  In fact, the reverse implication holds for all trees with $n\leq 8$, and both invariants provide strong distinguishing power.  For $n\leq 10$, the equivalence classes induced by $\gdpop$ and $\hdpop$ are all singletons; that is, both the GDP and the HDP are complete invariants.  For $11\leq n\leq 18$, the non-singleton equivalence classes are all of size two, and are enumerated as follows:
\[\begin{array}{|c|cccccccc|} \hline
n & 11&12&13&14&15&16&17&18\\ \hline
\text{Number of size-two equivalence classes} & 1&1&1&5&1&7&19&15\\ \hline
\end{array}\]
Michael Tang [personal communication to the authors] discovered two 19-vertex trees $T_1,T_2$ with different GDPs but the same HDP, shown in Figure~\ref{fig:tang}.  The equality of the HDPs may be checked computationally.  To observe without machine computation that the GDPs differ, one may observe that the five circled vertices of $T_1$ form a coclique incident to a total of 15 edges; no such coclique occurs in $T_2$.  Thus the coefficient of $x^5y^{15}$ is nonzero in $\hdp{T_1}$ (in fact it is 1), while it is zero in $\hdp{T_2}$.

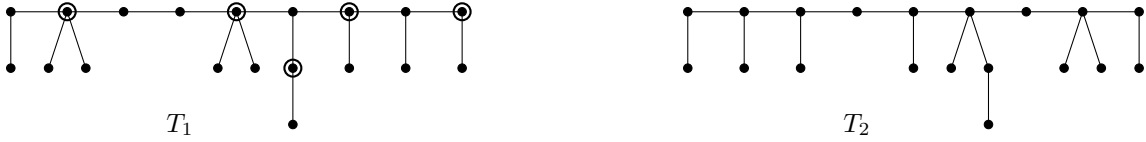
\begin{figure}[th]
\begin{center}
\begin{tikzpicture}[scale=0.75]
\coordinate (t1v00) at (5,0);
\coordinate (t1v01) at (6,0);
\coordinate (t1v02) at (7,0);
\coordinate (t1v03) at (8,0);
\coordinate (t1v04) at (9,0);
\coordinate (t1v05) at (6,-1);
\coordinate (t1v06) at (4,0);
\coordinate (t1v07) at (3,0);
\coordinate (t1v08) at (2,0);
\coordinate (t1v09) at (1,0);
\coordinate (t1v10) at (5.33,-1);
\coordinate (t1v11) at (4.67,-1);
\coordinate (t1v12) at (7,-1);
\coordinate (t1v13) at (8,-1);
\coordinate (t1v14) at (9,-1);
\coordinate (t1v15) at (6,-2);
\coordinate (t1v16) at (2.33,-1);
\coordinate (t1v17) at (1.67,-1);
\coordinate (t1v18) at (1,-1);
\foreach \coo in {t1v00,t1v01,t1v02,t1v03,t1v04,t1v05,t1v06,t1v07,t1v08,t1v09,t1v10,t1v11,t1v12,t1v13,t1v14,t1v15,t1v16,t1v17,t1v18}
	\draw[fill=black] (\coo) circle (.075);
\foreach \coo in {t1v00,t1v02,t1v04,t1v05,t1v08} \draw[thick] (\coo) circle (.15);
\draw (t1v18)--(t1v09)--(t1v08)--(t1v07)--(t1v06)--(t1v00)--(t1v01)--(t1v02)--(t1v03)--(t1v04)--(t1v14) (t1v16)--(t1v08)--(t1v17) (t1v10)--(t1v00)--(t1v11) (t1v01)--(t1v05)--(t1v15) (t1v03)--(t1v13) (t1v02)--(t1v12);
\node at (4,-2) {$T_1$};
\begin{scope}[shift={(12,0)}]
\coordinate (t2v00) at (5,0);
\coordinate (t2v01) at (6,0);
\coordinate (t2v02) at (7,0);
\coordinate (t2v03) at (8,0);
\coordinate (t2v04) at (9,0);
\coordinate (t2v05) at (6.33,-1);
\coordinate (t2v06) at (4,0);
\coordinate (t2v07) at (3,0);
\coordinate (t2v08) at (2,0);
\coordinate (t2v09) at (1,0);
\coordinate (t2v10) at (5,-1);
\coordinate (t2v11) at (5.67,-1);
\coordinate (t2v12) at (7.67,-1);
\coordinate (t2v13) at (8.33,-1);
\coordinate (t2v14) at (9,-1);
\coordinate (t2v15) at (6.33,-2);
\coordinate (t2v16) at (3,-1);
\coordinate (t2v17) at (2,-1);
\coordinate (t2v18) at (1,-1);
\foreach \coo in {t2v00,t2v01,t2v02,t2v03,t2v04,t2v05,t2v06,t2v07,t2v08,t2v09,t2v10,t2v11,t2v12,t2v13,t2v14,t2v15,t2v16,t2v17,t2v18}
	\draw[fill=black] (\coo) circle (.075);
\draw (t2v14)--(t2v04)--(t2v03)--(t2v02)--(t2v01)--(t2v00)--(t2v06)--(t2v07)--(t2v08)--(t2v09)--(t2v18) (t2v07)--(t2v16) (t2v08)--(t2v17) (t2v00)--(t2v10) (t2v12)--(t2v03)--(t2v13) (t2v11)--(t2v01)--(t2v05)--(t2v15);
\node at (4,-2) {$T_2$};
\end{scope}
\end{tikzpicture}
\end{center}
\caption{Two trees $T_1,T_2$ with the same HDP but different GDPs, found by Michael Tang.}\label{fig:tang}
\end{figure}

\subsection{The \texorpdfstring{$\mathcal{L}$}{L}-polynomial}
The $\mathcal{L}$-polynomial of a composition $\alpha$
is defined as 
\[\mathcal{L}(\alpha)=\sum_{\gamma\geq \alpha}\prod_{i}x_{\gamma_i}.\]
It is a specialization of the $U$-polynomial and hence also a specialization of the chromatic symmetric function \cite{APZ}. It is clear that $\Htwo(\alpha)$ can be computed from $\mathcal{L}(\alpha)$.  In \cite{BTvW,APZ}, it was shown that two caterpillars have the same $\mathcal{L}$-polynomial if and only if their signatures are related by switching.
\begin{question}
If $\Htwo(\alpha)=\Htwo(\beta)$, are $\alpha$ and $\beta$ necessarily related by switching?
\end{question}

We have verified by explicit computation that the answer to the latter question is affirmative for caterpillars up to 18 vertices.

\subsection{A closed formula for the gdp of a caterpillar}
The generalized degree polynomial of a caterpillar can be expressed compactly by grouping the sets $A\subseteq V(T)$ by $\spine(A) := A\cap\spine(T)$, since the contribution of each leaf to $\gdp{T}$ depends only on whether its neighbor belongs to $A$.  Thus we obtain
\begin{align}
\gdp{\Cat(\alpha)}
&= \sum_{U\subseteq[k]} x^{|U|} y^{\tilde d(U)} z^{\tilde e(U)} (xz+y)^{\sum_{u\in U}(a_u-1)} (xy+1)^{\sum_{u\in[k]\sm U}(a_u-1)}  \label{gdp-cat}\\
&= (xy+1)^{n-k}\sum_{U\subseteq[k]} x^{|U|} y^{d(U)} z^{e(U)} \left(\frac{xz+y}{xy+1}\right)^{\sum_{u\in U}(a_u-1)}\notag
\end{align}
where $\tilde d(U)$ and $\tilde e(U)$ are the numbers of boundary edges and internal edges in the spine, i.e.,
\begin{align*}
\tilde d(U) &= |\{i\in[k-1]\colon i\in U \text{ xor } i+1 \in U\}|,\\ 
\tilde e(U) &= |\{i\in[k-1]\colon i\in U \text{ and } i+1 \in U\}|.
\end{align*}
There is a similar formula for the half-generalized degree polynomial of a caterpillar, since $A\subseteq V(T)$ induces a subtree if and only if either (i) $A=\{v\}$ for some $v\in L(T)$, or (ii) $\spine(A)$ is a nontrivial path and every non-spine vertex in $A$ has a neighbor in $\spine(A)$.  Thus we obtain
\begin{equation} \label{hdp-cat}
\hdp{\Cat(\alpha)} = (n-k)y + \sum_{1\leq i\leq j\leq k} y^{i>1} y^{j<k} z^{j-i} (z+y)^{\sum_{u=i}^j(a_u-1
)}
\end{equation}
where, e.g., the symbol $y^{i>1}$ is interpreted as ``$y$ if $i>1$, else 1.''  Formulas~\eqref{gdp-cat} and~\eqref{hdp-cat} are useful for explicit computation; for instance, the number of terms in~\eqref{gdp-cat} is exponential in the size of the spine rather than the size of the tree, a significant savings.  On the other hand, we do not know how to obtain usable recurrences for these expressions.

\subsection{Combining the HDP and the STP}

We propose another tree invariant for further study.
Define the \defterm{souped-up subtree polynomial} as the following common generalization of the HDP and STP:
\[\soup{T}=\soup{T}(x,\,y,\,z)=\sum_{S\in\SSS(T)} x^{e(S)} y^{d(S)} z^{\ell(S)}\]
so that $\hdp{T}=\soup{T}(z,\,y,\,1)$ and $\stp{T}=\soup{T}(q,\,1,\,r)$.
For instance, the souped-up subtree polynomials of the trees $P_4$ and $S_4$ are
\begin{align*}
\soup{P_4} &= x^3z^2 + 2x^2yz^2 + xy^2z + 2xyz + 2y^2 + 2y,\\
\soup{S_4} &= x^3z^3 + 3x^2yz^2 + 3xy^2z + y^3 + 3y
\end{align*}
(we omit the details of the computation).

\begin{question}
What can be said about the ability of $\soup{T}$ to distinguish trees?
\end{question}

The souped-up subtree polynomial contains strictly more information than either the half-generalized degree polynomial or the subtree polynomial, because it distinguishes the two 11-vertex caterpillars in Figure~\ref{fig:sameGDP}.  Indeed, $T_1$ has exactly one subgraph with three edges, two leaves, and one external edge (induced by the four leftmost vertices), but $T_2$ has none.  Therefore, the coefficients of $x^3yz^2$ differ in $\soup{T_1}$ and $\soup{T_2}$.  In particular, $\soup{T}$ cannot be recovered from the generalized degree polynomial.

We have not found a pair of non-isomorphic trees with the same souped-up subtree polynomial, and we have checked by explicit computation that $\soup{T}$ is a complete invariant for trees with 18 or fewer vertices.  We have also checked computationally that neither $\soup{T}$ nor $\csf{T}$ can be obtained from the other linearly for $n=8$.

\bibliographystyle{amsalpha}
\bibliography{biblio}
\end{document}